\documentclass[10pt,a4paper]{article}
\usepackage[czech,english]{babel}
\usepackage[cp1250]{inputenc}
\usepackage{amssymb,amsfonts,amsmath,amsthm}
\usepackage{latexsym}
\usepackage{graphicx}
\usepackage{fullpage}
\usepackage{pstricks}
\usepackage{hyperref}
\usepackage[all]{xy}
\newcommand{\ie}{\emph{i.e.}}
\newcommand{\eg}{\emph{e.g.}}
\newcommand{\cf}{\emph{cf}}
\newcommand{\etc}{\emph{etc}}

\def\H{\mathcal{H}}
\def\eps{\varepsilon}
\def\Heps{H_{\eps}}
\def\R{\mathbb{R}}

\def\N{\mathbb{N}}
\def\Z{\mathbb{Z}}
\def\Ckappa{\|\kappa\|_{\infty}}
\def\Dom{\mathrm{Dom}\,}

\def\Qeps{Q_{\eps}}

\def\Re{\mathrm{Re}}
\def\intom{\int_{\Omega_0}}
\def\B{\mathcal{B}}
\def\heps{h_{\eps}}
\def\der{(\partial_s+\dot{\theta}\partial_{\alpha})}
\newcommand{\ader}{\partial_\alpha}
\def\keps{k^{\eps}}
\def\Ueps{U_{\eps}}
\def\Heff{H_{\mathrm{eff}}}
\def\qeff{q_{\mathrm{eff}}}
\def\Qeff{Q_{\mathrm{eff}}}

\newtheorem{defi}{Definition}[section]
\newtheorem{theorem}[defi]{Theorem}
\newtheorem{lemma}[defi]{Lemma}
\newtheorem{remark}[defi]{Remark}
\newtheorem{proposition}[defi]{Proposition}
\newtheorem{example}[defi]{Example}
\newtheorem{assumption}{Assumption}

\numberwithin{equation}{section}

\begin{document}
\title{\textbf{\LARGE
The effective Hamiltonian in curved quantum waveguides
under mild regularity assumptions
}}
\author{
David Krej\v{c}i\v{r}{\'\i}k\,$^{1,2}$
\ and \
Helena \v{S}ediv\'{a}kov\'{a}$^{1,3}$
}
\date{\small
%
%\begin{center}
\begin{quote}
\begin{enumerate}
\item[$1$]
\emph{Department of Theoretical Physics,
Nuclear Physics Institute ASCR,
25068 \v Re\v z, Czech Republic};
krejcirik@ujf.cas.cz, sedivakova.h@gmail.com
%\medskip \\
\item[$2$]
\emph{IKERBASQUE, Basque Foundation for Science,
48011 Bilbao, Kingdom of Spain}
\item[$3$]
\emph{Faculty of Nuclear Sciences and Physical Engineering,
Czech Technical University in Prague,
B\v{r}ehov\'{a} 7, 115 19 Prague 1, Czech Republic}
\end{enumerate}
\end{quote}
%\end{center}
%
\ \smallskip \\
6 March 2012
}
\maketitle
\begin{abstract}
\noindent
The Dirichlet Laplacian in a curved three-dimensional tube built
along a spatial (bounded or unbounded) curve is investigated in the limit
when the uniform cross-section of the tube diminishes.
Both deformations due to bending and twisting of the tube are considered.
We show that the Laplacian converges in a norm-resolvent sense
to the well known one-dimensional Schr\"odinger operator
whose potential is expressed in terms of the curvature of the reference curve,
the twisting angle and a constant measuring the asymmetry of the cross-section.
Contrary to previous results, we allow the reference curves
to have non-continuous and possibly vanishing curvature.
For such curves, the distinguished Frenet frame
standardly used to define the tube need not exist
and, moreover, the known approaches to prove the result
for unbounded tubes do not work.
Our main ideas how to establish the norm-resolvent convergence
under the minimal regularity assumptions
are to use an alternative frame defined
by a parallel transport along the curve and a refined
smoothing of the curvature via the Steklov approximation.
\end{abstract}
\newpage
\tableofcontents
\newpage
%
%---------------------%
\section{Introduction}
%---------------------%
%
This paper is concerned with the singular operator limit
for the Dirichlet Laplacian
in a three-dimensional non-self-intersecting
curved tube (\cf~Figure~\ref{vlnovod})
when its two-dimensional cross-section shrinks to a point.
The tube~$\Omega_\varepsilon$ is constructed by translating and rotating
the cross-section along a spatial curve~$\Gamma$
and the limit is realized by homothetically
scaling a fixed cross-section~$\omega$
by a small positive number~$\varepsilon$.
Without loss of generality, we assume that the curve
is given by its arc-length parameterization $\Gamma:I\to\R^3$,
where the open interval $I \subset \R$ is allowed
to be arbitrary: finite, infinite or semi-infinite.
Geometrically, $\Omega_\varepsilon$ collapses to~$\Gamma$ as $\varepsilon\to 0$.
We are interested in how and when the three-dimensional
Dirichlet Laplacian $-\Delta_D^{\Omega_\varepsilon}$
can be approximated by a one-dimensional operator~$H_\mathrm{eff}$
on the curve.

\begin{figure}[h]
\begin{center}
\includegraphics[width=10 cm]{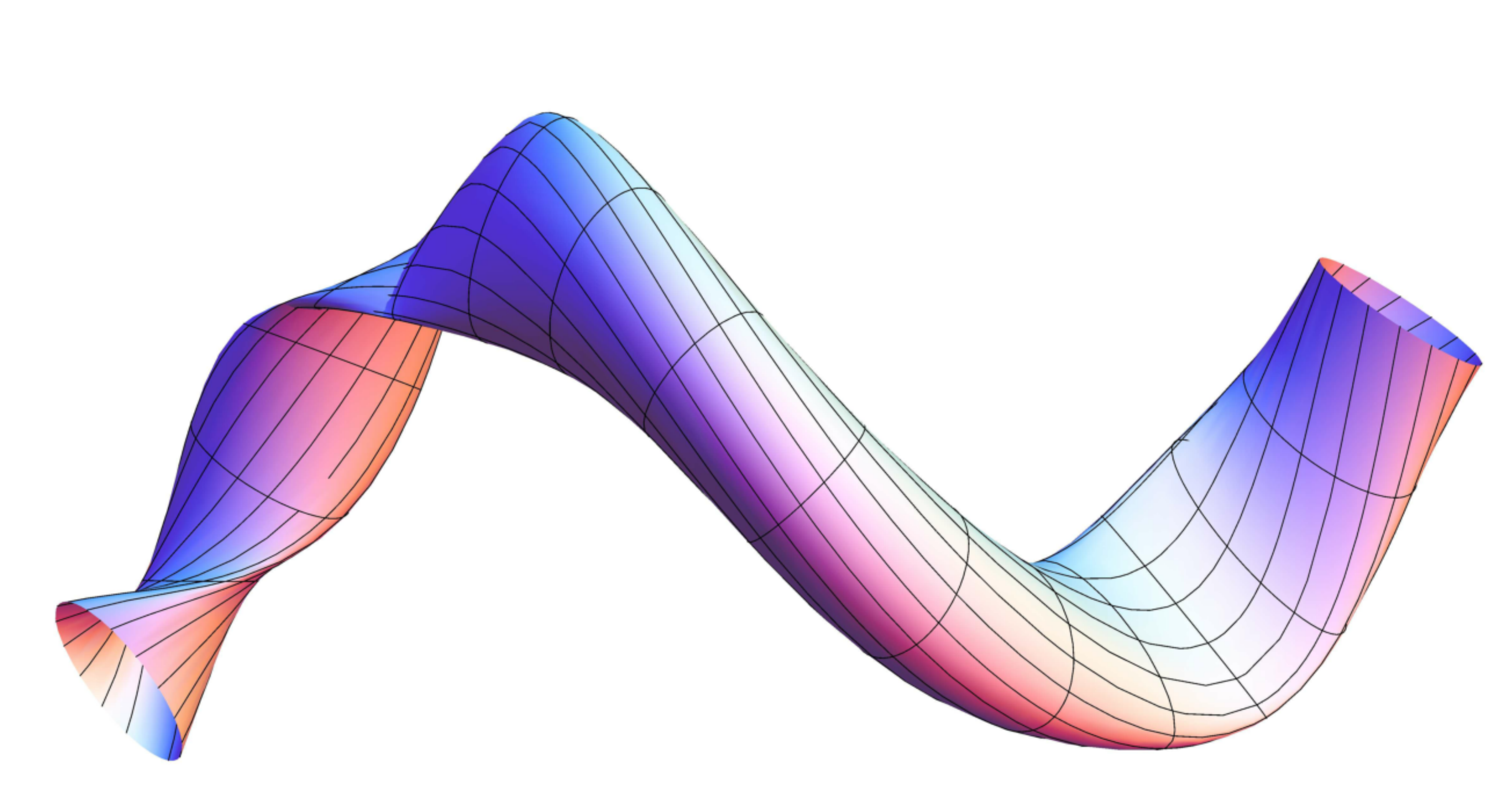}
\caption{The geometry of a quantum waveguide.
Twisting and bending
are demonstrated on the left and right part of the figure,
respectively.}\label{vlnovod}
\end{center}
\end{figure}

We start with some more or less obvious observations.
\begin{itemize}
\item[$\circ$]
Since we deal with unbounded operators, the convergence
of $-\Delta_D^{\Omega_\varepsilon}$ to~$H_\mathrm{eff}$
is understood through a convergence of their resolvents.
\item[$\circ$]
The Dirichlet boundary conditions imply that the spectrum
of $-\Delta_D^{\Omega_\varepsilon}$ explodes as $\varepsilon \to 0$.
It is just because the first eigenvalue of
the Dirichlet Laplacian in the scaled cross-section
$
  \varepsilon\omega := \{\eps t \, | \, t\in\omega\}
$
equals $\varepsilon^{-2} E_1$, where~$E_1$ is the first eigenvalue
of the Dirichlet Laplacian in the fixed cross-section~$\omega$.
Hence, a normalization $-\Delta_D^{\Omega_\varepsilon}-\varepsilon^{-2} E_1$
is in order to get a non-trivial limit.
\item[$\circ$]
Finally, since the configuration spaces $\Omega_\varepsilon$
and $\Gamma$ have different dimensions, a suitable identification
of respective Hilbert spaces of $-\Delta_D^{\Omega_\varepsilon}$
and $H_\mathrm{eff}$ is required.
This is achieved by using a unitary transform
that identifies $L^2(\Omega_\varepsilon)$ with $L^2(I\times\omega)$
and by considering~$H_\mathrm{eff}$
as acting on the subspace of $L^2(I\times\omega)$
spanned by functions of the form $\varphi\otimes\mathcal{J}_1$
on $I \times \omega$, where~$\mathcal{J}_1$ denotes
the positive normalized eigenfunction of~$-\Delta_D^{\omega}$
corresponding to~$E_1$.
\end{itemize}

Taking these remarks into account, we can
write the convergence result as follows:
\begin{equation}\label{limit}
  -\Delta_D^{\Omega_\varepsilon} -\varepsilon^{-2} E_1
  \ \xrightarrow[]{\varepsilon \to 0} \
  H_\mathrm{eff}^F := -\Delta_D^I
  - \frac{\kappa^2}{4} + C_\omega \, (\dot\theta_F-\tau)^2
  \,.
\end{equation}
Here $\kappa:=|\ddot{\Gamma}|$ and
$
  \tau :=
  %\kappa^{-2}\;\!\dot{\Gamma}\cdot\ddot{\Gamma}\times\dddot{\Gamma}
  \kappa^{-2}\det(\dot{\Gamma},\ddot{\Gamma},\dddot{\Gamma})
$
denote respectively the curvature and torsion of~$\Gamma$,
$\theta_F$~is an angle function defining the rotation
of~$\varepsilon\omega$ with respect to the Frenet frame of~$\Gamma$
and $C_\omega := \|\ader\mathcal{J}_1\|_{L^2(\omega)}$,
with~$\ader$ denoting the angular derivative in~$\R^2$.

The convergence~\eqref{limit} can be employed
as a way to approximate the three-dimen\-sion\-al
dynamics of an electron constrained to a curved quantum waveguide
by the effective one-dimensional Hamiltonian~$H_\mathrm{eff}^F$
on the reference curve.
The Dirichlet Laplacian on the interval~$I$
represents the kinetic energy of the free motion on the reference curve
(indeed, $-\Delta_D^I$ is unitarily equivalent
to the Laplace-Beltrami operator on~$\Gamma$).
The additional potential of~$H_\mathrm{eff}^F$ clearly consists of
two competing terms:
the negative one induced by curvature
and the positive one due to torsion.
They respectively represent the opposite effects
of bending and twisting in quantum waveguides,
\emph{cf}~\cite{K6-with-erratum}.

\subsection{Known results and why we write this paper}\label{Sec.why}
The result~\eqref{limit} is well known,
it has been established in various settings
and with different methods during the last two decades.
As the first rigorous result, let us mention
the classical paper~\cite{DE} of Duclos and Exner,
where the norm-resolvent convergence of~\eqref{limit}
is proved under quite restrictive hypotheses $\Gamma \in C^4$
and~$\omega$ being a disc (so that $C_\omega=0$).
More precise results about the limit
(for instance, uniform convergence of eigenfunctions)
in arbitrary dimensions are established
by Freitas and Krej\v{c}i\v{r}\'ik in~\cite{FK4},
but the cross-section is still assumed
to be rotated along~$\Gamma$ in such a way that
$\dot\theta_F=\tau$, so there is no effect of twisting.

The presence of the additional potential term due to twisting
in~$H_\mathrm{eff}^F$ was observed for the first time
by Bouchitt\'e, Mascarenhas and Trabucho~\cite{BMT}.
Contrary to the previous works where operator techniques are used,
the authors of~\cite{BMT} use an alternative method of
Gamma-convergence, which provides just a strong-resolvent
convergence of~\eqref{limit} but, on the other hand,
enables them to weaken the regularity hypothesis to $\Gamma \in C^3$.
De Oliveira \cite{deOliveira_2006} extended the results of~\cite{BMT}
to unbounded tubes and established a norm-resolvent convergence
in the bounded case
(see also \cite{deOliveira-Verri1,deOliveira-Verri2}).

Finally, let us mention the series of recent papers
\cite{Lampart-Teufel-Wachsmuth,Wachsmuth-Teufel,Wachsmuth-Teufel-short},
where the singular limit of the type~\eqref{limit}
is attacked by the methods of adiabatic perturbation theory.
In fact, the general setting of shrinking tubular neighbourhoods of
(infinitely smooth) submanifolds of Riemannian manifolds
is considered in these works
and the results can be interpreted as a rigorous quantization procedure
on the submanifolds.

After having provided an extensive literature on the limit~\eqref{limit},
a question arises why we still consider the problem in the present paper.
In fact, the issue we would like to address here is about
the optimal regularity conditions under which
the effective approximation~\eqref{limit} holds.
We are motivated by the fact that the known existing
results mentioned above do not cover physically interesting
curves with merely continuous or even discontinuous curvature
(\cf~Figure~\ref{nefrenet}).

\begin{figure}[h]
\begin{center}
\includegraphics[width=15 cm]{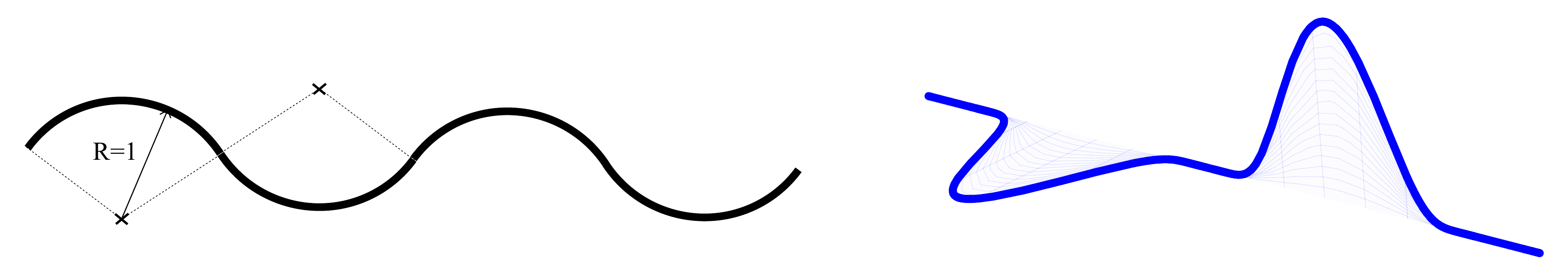}
\caption{Examples of curves with discontinuous curvature (on the left)
and with infinitely smooth curvature but still
without the Frenet frame (on the right).}\label{nefrenet}
\end{center}
\end{figure}

Furthermore, it is a standard hypothesis in the literature
about quantum waveguides that the first three derivatives
of the reference curve~$\Gamma$
exist and are linearly independent, so that the torsion
and the distinguished Frenet frame exist.
However, this is meaningful only for curves which are three
times differentiable and have nowhere vanishing
(differentiable) curvature~$\kappa$.
We find the latter as a very restrictive requirement,
even for infinitely smooth curves (\cf~Figure~\ref{nefrenet}).
Indeed, the torsion~$\tau$ is not well defined
for such curves, so that the limit~\eqref{limit}
with~the effective Hamiltonian~$H_\mathrm{eff}^F$ is meaningless.
Partial attempts to overcome this technical condition
can be found in \cite{Borisov-Cardone,ChDFK}.
In this paper we provide a complete answer by considering
waveguides built along any twice differentiable curves,
with the boundedness of~$\kappa$ being the only hypothesis.
Our assumptions are very natural and in fact
intrinsically necessary for the construction of the waveguide
as a regular Riemannian manifold.

Finally, the Gamma-convergence method of \cite{BMT,deOliveira_2006},
which seems to work under less restrictive regularity once the technical
difficulty of the non-existence of the Frenet frame is overcome,
implies only (unless the waveguide is bounded \cite{deOliveira_2006})
a strong-resolvent convergence for~\eqref{limit}.
Furthermore, it does not provide any information
about the convergence rate.
In addition to the regularity issues mentioned above,
our goal is therefore to use operator methods instead of the Gamma-convergence,
establish~\eqref{limit} in the norm-resolvent sense
and get a control on the convergence rate.

\subsection{The content of the paper}
The organization of this paper is as follows.

In the following Section~\ref{Sec.strategy.intro}
we explain our strategy to handle the singular limit
under mild regularity hypotheses
and state the main result of this paper
(Theorem~\ref{Thm.main.intro}).

We postpone a precise definition of a simultaneously
twisted and bent waveguide
and of the associated Dirichlet Laplacian till Section~\ref{Sec.pre}.
For reasons mentioned above, we construct the waveguide
by using an alternative frame defined by parallel transport
along the curve instead of the usual Frenet frame.
Since it seems that this frame is not as well known
as the Frenet one, and since we want to include
more general curves than those usually considered in differential geometry,
we decided to include Section~\ref{secframe},
where we thoroughly describe the construction of the frame
under our mild regularity conditions.

The main idea of the present paper consists in smoothing
non-differentiable quantities by means of the
so-called Steklov approximation (see~\eqref{aproximace} below).
This procedure is in detail explained in Section~\ref{Sec.Steklov}.

The proof of Theorem~\ref{Thm.main.intro} is given
in Section~\ref{Sec.proof}.
Since it is rather long and technically involved,
we divide the proof into several auxiliary lemmata
and the section into corresponding subsections.

The paper is concluded in Section~\ref{seconcl}
by discussing optimality of our results.

%-----------------------------------------%
\section{Our strategy and the main result}\label{Sec.strategy.intro}
%-----------------------------------------%
%
Our strategy how to achieve the objectives sketched in Introduction
is based on the following ideas:
\begin{enumerate}
\item[(I)]
Use the frame defined by the parallel transport
instead of the Frenet frame. This alternative frame is known
to exist for any curve of class~$C^2$, \emph{cf}~\cite{Bishop_1975}.
We generalize the construction to the curves
that merely belong to the Sobolev space $W_\mathrm{loc}^{2,\infty}$.
\item[(II)]
Work exclusively with the quadratic forms associated with the operators.
More specifically, we adapt the elegant method of Friedlander and Solomyak
\cite{Friedlander-Solomyak_2007,Friedlander-Solomyak_2008a}
to deduce the norm-resolvent convergence
from a convergence of quadratic forms.
\end{enumerate}

Even if one implements these ideas, the standard operator
approach to the thin-cross-section limit in quantum waveguides
(see, \eg, \cite{DE})
still requires certain differentiability of curvature~$\kappa$
(which is just bounded under our hypotheses).
To see it, we sketch the standard strategy now.

First, one uses curvilinear coordinates,
which induce the unitary transform
\begin{equation}\label{U1}
  U_1 : L^2(\Omega_\varepsilon)
  \to L^2\big(
  I\times\omega,\varepsilon^2\;\!h(s,t)\;\!ds\;\!dt
  \big),
\end{equation}
where the Jacobian~$\varepsilon^2 h$ is standardly expressed
in terms of~$\kappa$ and~$\theta_F$.
In our more general setting enabled by the strategy~(I) above,
we have
\begin{equation}\label{Jacobian}
  h(\cdot,t)
  := 1 - \varepsilon \, t_1 \, (k_1\cos\theta - k_2 \sin\theta)
  - \varepsilon \, t_2\, (k_1 \sin\theta + k_2 \cos\theta)
  \,,
\end{equation}
where~$k_1,k_2$ are curvature functions computed with respect to
our relatively parallel frame and~$\theta$ is an angle function
defining the rotation of the cross-section~$\varepsilon\omega$
with respect to this frame.
We have $\kappa^2=k_1^2+k_2^2$ and, if the Frenet frame exists,
our frame is rotated with respect to the Frenet frame
by the angle given by a primitive of torsion~$\tau$
(\cf~\eqref{anglefrenetRPAF} below).
Consequently, in our more general setting,
the difference $\dot\theta_F-\tau$ in~\eqref{limit}
is to be replaced by~$\dot\theta$
and the effective Hamiltonian reads
\begin{equation}\label{defheff}
  \Heff := - \Delta_D^{I}
  - \frac{\kappa^2}{4} + C_\omega \;\! \dot{\theta}^2
  .
\end{equation}
We emphasize that this operator coincides with~$H_\mathrm{eff}^F$
introduced in~\eqref{limit} if~$\Gamma$ possesses the Frenet frame
but, contrary to~$H_\mathrm{eff}^F$,
it is well defined even if the torsion~$\tau$ does not exist.

Second, to recover the curvature term in the effective
potential of~\eqref{limit},
one also performs the unitary transform
\begin{equation}\label{U2}
  U_2 : L^2\big(
  I\times\omega,\varepsilon^2\;\!h(s,t)\;\!ds\;\!dt
  \big)
  \to L^2(I\times\omega) : \left\{ \psi \mapsto \varepsilon \;\! \sqrt{h} \;\! \psi \right\}
  \,.
\end{equation}
The composition $U:=U_2 U_1$ clearly identifies
the geometrically complicated Hilbert space
$L^2(\Omega_\varepsilon)$ with the simple $L^2(I\times\omega)$.
The standard procedure consists in transforming
$-\Delta_D^{\Omega_\eps}$ with help of~$U$ to
a unitarily equivalent operator on $L^2(I\times\omega)$
and prove the norm-resolvent convergence for the transformed operator.
However, $U$~does not leave the form domain
$W_0^{1,2}(I\times\omega)$ invariant if~$k_1,k_2$
are not differentiable in a suitable sense.

The last difficulty is overcome in this paper
by the following trick:
\begin{itemize}
\item[(III)]
Replace the curvature functions in~\eqref{Jacobian} by their
$\varepsilon$-dependent mollifications ($\mu\in\{1,2\}$)
\begin{equation}\label{aproximace}%\label{Steklov.approx}
  k_\mu^\varepsilon(s) :=
  \frac{1}{\delta(\varepsilon)}
  \int_{s-\frac{\delta(\varepsilon)}{2}}^{s+\frac{\delta(\varepsilon)}{2}}
  k_\mu(\xi) \, d\xi
  \,,
\end{equation}
where~$\delta$ is a continuous function such that
both $\delta(\eps)$ and $\varepsilon\;\!\delta(\varepsilon)^{-1}$
tend to zero as $\varepsilon \to 0$.
\end{itemize}
Then everything works very well
(although the overall procedure is technically much more demanding)
because the longitudinal derivative
of the mollified~$h$ involves the terms
$\varepsilon \dot{k}_\mu^\varepsilon$ which vanish
as $\varepsilon \to 0$,
even if $\dot{k}_\mu^\varepsilon$ diverge in this limit.
In more intuitive words,~\eqref{aproximace}~can be
understood in a sense as that
the curve is smoothed on a scale small compared
to the curvature of the curve,
but large compared to the diameter of the cross-section
of the waveguide.

The mollification~\eqref{aproximace} is sometimes referred
to as the Steklov approximation in Russian literature
(see, \eg, \cite{Akhiezer}).
At a step of our proof, we shall also need to mollify
the derivative of the angle function~$\theta$.

Before stating the main result of the paper,
let us now carefully write down all the hypotheses
we need to derive it, although some of the quantities
appearing in the assumptions will be properly defined only later.
\begin{assumption}\label{assumgamma}
Let $\Gamma:I \rightarrow \R^3$ be a unit-speed spatial curve,
where the interval $I\subset\R$ is finite, semi-infinite or infinite,
satisfying
\begin{itemize}
\item[\emph{(i)}]
$\Gamma\in W^{2,\infty}_{\mathrm{loc}}(I;\R^3)$
\ and \
$\kappa:=|\ddot{\Gamma}|\in L^\infty(I)$.
\end{itemize}
Further, let~$\omega$ be a bounded open connected subset
of~$\R^2$ and let $\theta:I\to\R$ be the angle describing
the rotation of the waveguide cross-section $\eps\omega$
with respect to the relatively parallel adapted frame
constructed along~$\Gamma$ satisfying
\begin{itemize}
\item[\emph{(ii)}]
$\theta\in W^{1,\infty}_{\mathrm{loc}}(I)$
\ and \
$\dot{\theta}\in L^\infty(I)$.
\end{itemize}
Finally, we assume
\begin{itemize}
\item[\emph{(iii)}] $\Omega_\eps$ does not overlap itself
for all sufficiently small~$\eps$.
\end{itemize}
\end{assumption}

The conditions stated in Assumption~\ref{assumgamma}
are quite week and in fact very natural
for the construction of the waveguide~$\Omega_\eps$
and for obtaining reasonable spectral consequences from~\eqref{limit}
(\cf~Section~\ref{seconcl} for further discussion).
Unfortunately, for making our strategy to work
in the case of unbounded waveguides,
we also need to assume the following (seemingly technical) hypothesis.
\begin{assumption}\label{ass2}
For any $f \in L_\mathrm{loc}^\infty(I)$, let us define
\begin{equation}\label{sigmaf}
\sigma_f(\delta(\eps))
:= \sup_{n\in\Z} \sqrt{
\sup_{|\eta|\leq \frac{\delta(\eps)}{2}} \int_{n}^{n+1} \
\left|f(s)-f(s+\eta)\right|^2 ds
} \,.
\end{equation}
where $\eps\mapsto\delta(\eps)$
is some continuous function vanishing with $\eps$.
To give a meaning to~\eqref{sigmaf} for $I\not=\R$,
we assume that~$f$ is extended from~$I$ to~$\R$ by zero.
We make the following two hypotheses
\begin{align}
\label{sigmakknule}
&\lim_{\eps\to 0} \sigma_k(\delta(\eps))
:= \lim_{\eps\to 0} \sum_{\mu=1,2}\sigma_{k_{\mu}}(\delta(\eps)) = 0
\,,
\\
\label{sigmathetaknule}
&\lim_{\eps\to 0} \sigma_{\dot{\theta}} (\tilde{\delta}(\eps)) =0
\,.
\end{align}
for some positive continuous functions $\delta,\tilde{\delta}$
satisfying
\begin{equation}\label{delta.limits}
  \lim_{\eps\to 0} \delta(\eps) = 0
  \,, \qquad
  \lim_{\eps\to 0} \frac{\eps}{\delta(\eps)} = 0
  \,,\qquad
  \lim_{\eps\to 0} \tilde{\delta}(\eps) = 0
  \,.
\end{equation}
\end{assumption}

Assumption~\ref{ass2} is satisfied for a wide class of reference
curves~$\Gamma$ and rotation angles~$\theta$.
First of all, let us emphasize that
it always holds whenever~$I$ is bounded.
Indeed, this is a consequence of the more general fact
that Assumption~\ref{ass2} holds provided that
(the extensions of) the representants of~$f$
are square-integrable functions on~$\R$.
As other sufficient conditions which guarantee
the validity of Assumption~\ref{ass2},
let us mention that it holds whenever the representants
are either Lipschitz, or just uniformly continuous,
or periodic, \emph{etc}.
In any case, it is a non-void hypothesis for unbounded~$I$ only,
when it becomes important to have a control
over the behaviour of~$k_1,k_2$ and~$\dot\theta$ at infinity.

Now we are in a position to state the main result of this paper.
\begin{theorem}\label{Thm.main.intro}
Let Assumption~\ref{assumgamma} and Assumption~\ref{ass2} hold true.
Then there exist positive constants~$\eps_0$ and~$C$
such that for all $\eps \leq \eps_0$,
\begin{multline}\label{NR.bound}
  \left\|
  U(-\Delta_D^{\Omega_\varepsilon} -\varepsilon^{-2} E_1-i)^{-1}U^{-1}
  - (H_\mathrm{eff}-i)^{-1} \oplus 0^\bot
  \right\|_{\mathcal{B}(L^2(I\times\omega))}
  \\
  \leq C \left(
  \eps + \eps \;\! \|\dot{k}_1^\eps\| + \eps \;\! \|\dot{k}_1^\eps\|
  + \sigma_k(\delta(\eps)) + \sigma_{\dot\theta}(\tilde{\delta}(\eps))
  \right)
  \,,
\end{multline}
where~$0^\bot$ denotes the zero operator on the orthogonal
complement of the span of
$
  \{\varphi\otimes\mathcal{J}_1 \,|\, \varphi \in L^2(I)\}
$ and $U=U_2U_1$ is the unitary transform composed of~\eqref{U1} and~\eqref{U2}.
\end{theorem}

Recall that the quantities
$\varepsilon \|\dot{k}_1^\varepsilon\|_\infty$
and $\varepsilon \|\dot{k}_2^\varepsilon\|_\infty$
from the right hand side of~\eqref{NR.bound}
tend to zero as $\eps \to 0$.
Hence Theorem~\ref{Thm.main.intro} indeed implies
the norm-resolvent convergence of the type~\eqref{limit}
and it covers all the known results, and much more.
Furthermore, the right hand side of~\eqref{NR.bound}
explicitly determines the decay rate of the convergence~\eqref{limit}
as a function of the regularity properties of~$k_1,k_2$ and~$\dot\theta$.
Again, it reduces to the well known (see, \eg, \cite{FK4})
$\eps$-type decay rate for (uniformly) Lipschitz~$k_1,k_2$ and~$\dot\theta$
(\cf~Section~\ref{seconcl}).

Assumption~\ref{ass2} actually requires that
the curvatures~$k_1,k_2$ and~$\dot\theta$ are not oscillating
too quickly at infinity (if~$I$ is unbounded).
We leave as an open problem whether
it is possible to have the norm-resolvent convergence
without this hypothesis.

We refer to Section~\ref{seconcl} for further discussion
of the optimality of Theorem~\ref{Thm.main.intro}.

%---------------------%
\section{Preliminaries}\label{Sec.pre}
%---------------------%
%
In the following subsection we introduce the notion
of relatively parallel adapted frame
for any weakly twice differentiable spatial curve.
It is then used to define the tube~$\Omega_\eps$ in Section~\ref{sectube},
while the associated Dirichlet Laplacian is eventually introduced
in Section~\ref{sechamil}.

\subsection{The relatively parallel adapted frame}\label{secframe}
We closely follow the approach of Bishop~\cite{Bishop_1975}
who introduced the relatively parallel adapted frame
for $C^2$-smooth curves. Indeed, the extension to curves which
are only weakly differentiable requires rather minimal modifications.

Given an open interval $I \subset \R$
(finite, infinite or semi-infinite),
let $\Gamma:I\to\R^3$ be a $C^1$-smooth immersion.
Without loss of generality,
we assume that the curve~$\Gamma$ is unit-speed,
\ie\ $|\dot\Gamma(s)|=1$ for all $s \in I$.
Then $T:=\dot{\Gamma}$ represents
a continuous tangent vector field of~$\Gamma$.

A \emph{moving frame} along~$\Gamma$ is
a triplet of differentiable vector fields
$e_i:I\to\R^3$, $i=1,2,3$,
which form a local orthonormal basis, \ie,
%
%\begin{equation}\label{defframe}
$
  e_i(s)\cdot e_j(s)=\delta_{ij}
  %\qquad \forall s\in I
$
%\end{equation}
%
for all $s\in I$.
%(``$\cdot$'' denotes the standard scalar product in $\R^3$).
We say that a moving frame is \emph{adapted} to the curve
if the members of the frame are either tangent of perpendicular to the curve.
The Frenet frame (if it exists)
is the most common example of an adapted frame,
however, in this paper the so-called \emph{relatively parallel
adapted frame} (RPAF) will be used instead of it
(since it always exists).

We say that a normal vector field~$M$ along~$\Gamma$
is \emph{relatively parallel} if its derivative is tangential,
\ie\ $\dot{M} \times T = 0$.
Such a field can be indeed understood as moved by parallel transport,
since it turns only whatever amount is necessary for it to remain normal,
so it is a close to being parallel as possible without losing normality.

The RPAF then consists of the unit tangent vector field~$T$
and two unit normal relatively parallel
and mutually orthonormal vector fields~$M_1, M_2$.
Let us note that for any relatively parallel normal vector field~$M$,
we have
$
  (|M|^2)^{\mbox{\normalsize$\cdot$}}
  = 2 \dot{M}\cdot M = 0
$.
At the same time,
$
  (M_1 \cdot M_2)^{\mbox{\normalsize$\cdot$}}
  = 0
$.
That is, the lengths of the relatively parallel normal vector fields
and the angle between them are preserved.
Consequently, the definition of RPAF makes sense and it indeed represents
an adapted moving frame.

The existence of RPAF for any $C^2$-smooth curve
is proved in~\cite{Bishop_1975}.
However, such a regularity implies that
the curvature $\kappa:=|\ddot{\Gamma}|$ is continuous,
which is still a too strong assumption for us.
Hence, here we provide an extension of the construction of RPAF
to curves which are merely $\Gamma\in W^{2,\infty}_\mathrm{loc}(I;\R^3)$.
This implies that the curvature~$\kappa$ is locally bounded only,
which does not restrict our results whatsoever,
since the stronger assumption $\kappa\in L^{\infty}(I)$
will have to be assumed for other reasons anyway.

\begin{proposition}[Existence of RPAF]\label{proprpnvf}
Let $\Gamma\in W^{2,\infty}_{\mathrm{loc}}(I;\R^3)$
be a unit-speed curve
and let~$M_1^0$ and~$M_2^0$ be two unit normal vectors
at a point $\Gamma(s_0)$ such that
$
  \{T(s_0),M_1^0,M_2^0\}
$
is an orthonormal basis of the tangent space $\mathbb{T}_{\Gamma(s_0)}\R^3$.
Then there exists a unique relatively parallel adapted frame $\{T,M_1,M_2\}$,
such that $M_1(s_0)=M_1^0$ and $M_2(s_0)=M_2^0$.
The vector fields in this frame are continuous
and their weak derivatives exist and are locally bounded.
\end{proposition}
\begin{proof}
By $\Gamma\in W^{2,\infty}_\mathrm{loc}(I;\R^3)$
we mean precisely that $\Gamma^i\in W^{2,\infty}_\mathrm{loc}(I)$ for $i=1,2,3$,
which yields that $\Gamma^i\in C^1(I)$
and the derivative $\dot{\Gamma}^i$
is locally Lipschitz continuous.
This allows us to introduce a continuous unit tangent vector field
$T:=\dot{\Gamma}$ as before and we know that the weak derivative
of~$T$ exists and is locally bounded.

It remains to find the two relatively parallel
normal vector fields $M_1, M_2$.
First of all, let us notice that the uniqueness is trivial:
the difference of two relatively parallel normal vector fields
is also relatively parallel, hence preserves the length.
So if two such coincide at one point,
their difference has constant length zero.

In the first step we find some auxiliary unit normal vector fields
satisfying the initial conditions, \ie\ the vector fields~$N_{1}, N_2$
satisfying $N_{\mu}\cdot T = 0$, $N_{\mu}\cdot N_\nu = \delta_{\mu\nu}$
and $N_{\mu}(s_0)=M_{\mu}^0$, with $\mu,\nu = 1,2$.
Such fields can be constructed locally by employing the continuity of~$T$
and local boundedness of~$\kappa$.
Explicitly, assuming without loss of generality that
one of the coefficients $T^1(s_0)$ or $T^2(s_0)$
is greater or equal to~$1/3$, we can choose, for instance,
$$
  N_1 := \left(
  \frac{-T^2}{\sqrt{(T^1)^2+(T^2)^2}},
  \frac{T^1}{\sqrt{(T^1)^2+(T^2)^2}},
  0
  \right)
  \,, \qquad
  N_2 := T \times N_1
  \,.
$$
By means of the fundamental theorem of calculus,
we can easily establish the inequality
\begin{equation}\label{fundamental}
  |T(s)-T(s_0)| \leq |s-s_0| \, \|\kappa\|_{L^\infty((s_0,s))}
\end{equation}
for every $s \in I$, which shows that $N_1,N_2$
are well defined in a bounded open interval~$J$ around~$s_0$.
From the dependence of the components of~$N_{\mu}$ on~$T$,
we deduce that both $N_{\mu} \in W^{1,\infty}(J;\R^3)$.

In the second step we have to realize that it is always possible
to find a continuous and in the weak sense differentiable function
$\vartheta:J\to\R$
satisfying $\vartheta(s_0)=0$ and such that the normal vector field
$M_1:=N_1\cos{\vartheta} + N_2\sin{\vartheta}$ is relatively parallel in~$J$.
This is easily established by expressing the derivative
of the triple $\{T,N_1,N_2\}$ by means of an antisymmetric
Cartan matrix and by choosing~$\vartheta$ as primitive
of the non-tangential coefficient of the matrix
coming from the derivatives of~$N_1,N_2$.
Then also $M_2:=-N_1\sin{\vartheta} + N_2\cos{\vartheta}$
is relatively parallel, both $M_{\mu} \in W^{1,\infty}(J;\R^3)$
and $M_\mu(s_0)=M_\mu^0$.

Finally, to get the global existence on~$I$,
we can patch together the local RPAFs,
which exist in a covering by bounded intervals
because of~\eqref{fundamental}.
The regularity at the points where they link together
is a consequence of the uniqueness part.
\end{proof}

If $\{T,M_1,M_2\}$ is a RPAF,
we have
\begin{equation}\label{cartan}
\left(
\begin{array}{c}
  T   \\
  M_1 \\
  M_2 \\
\end{array}
\right)^{\!\!\mbox{\large$\cdot$}} =
\left(
\begin{array}{ccc}
  0    & k_1 & k_2\\
  -k_1 & 0   & 0  \\
  -k_2 & 0   & 0  \\
\end{array}
\right)\left(
\begin{array}{c}
  T\\
  M_1 \\
  M_2\\
\end{array}
\right).
\end{equation}
Due to $\Gamma\in W^{2,\infty}_\mathrm{loc}(I;\R^3)$,
functions~$k_1$ and~$k_2$ are locally bounded,
however they do not need to be neither differentiable nor continuous.

Analogous functions for the Frenet frame,
\ie\ the curvature~$\kappa$ and torsion~$\tau$,
are uniquely determined for a non-degenerate curve (\ie~$\kappa>0$).
Let us examine the uniqueness of $k_1,k_2$ in our general situation.
Proposition~\ref{proprpnvf} says that for a given curve,
RPAF is unique if the initial vectors $M_1^0,M_2^0$
at some point~$s_0$ are specified.
For rotated initial vectors
$\tilde{M}_{\mu}^0 := \sum_{\nu = 1}^2\mathcal{R}_{\mu\nu}M_{\nu}^0$,
with $\mu=1,2$,
where $\mathcal{R}$ is any constant $2\times 2$ orthogonal matrix,
a different RPAF is obtained in general.
Consequently, the functions $k_{\mu}$ transfers to
$\tilde{k}_{\mu} = \sum_{\nu = 1}^2\mathcal{R}_{\mu\nu} k_{\nu}$,
with $\mu = 1,2$.
Hence the curvatures $k_1,k_2$ are not unique for the curve.

On the other hand, we have
\begin{equation}\label{defkappa}
  \kappa = |\dot{T}| = |k_1 M_1 + k_2 M_2|
  = \sqrt{k_1^2 + k_2^2}
  \,,
\end{equation}
hence the magnitude of the vector $(k_1,k_2)$
is independent of the choice of RPAF.
Finally, let us assume that the curve~$\Gamma$ possesses
the distinguished Frenet frame and let us denote by~$N$
the principal normal and by~$B$ the binormal.
It is easy to check that the pair of vectors $\{M_1,M_2\}$
is rotated with respect to $\{N,B\}$ by the angle
\begin{equation}\label{anglefrenetRPAF}
  \vartheta(s) = \vartheta_0 + \int_{s_0}^{s} \tau(\xi) \, d\xi
  \,,
\end{equation}
where $\vartheta_0$~is the angle between vectors $M_1(s_0)$ and $N(s_0)$.
Consequently, $\tau=\dot\vartheta$.
Writing, $(k_1,k_2)=(\kappa\cos\vartheta,\kappa\sin\vartheta)$,
we can conclude that~$\kappa$ and an indefinite integral of~$\tau$
represent polar coordinates for the curve $(k_1,k_2)$,
as pointed out in~\cite{Bishop_1975}.

In Figure~\ref{frenetvsrpaf} we can see
how the pair of Frenet normal vectors $\{N,B\}$
versus relatively parallel normal vectors $\{M_1,M_2\}$
move along a helix.
The longer side of the rectangular cross-section
corresponds to the direction of~$N$ (on the left)
and~$M_1$ (on the right),
whereas the shorter side is the direction of~$B$
and~$M_2$, respectively.
In the bottom of the Figure, the two frames coincide.

\begin{figure}[h]
\begin{center}
\includegraphics[width=10 cm]{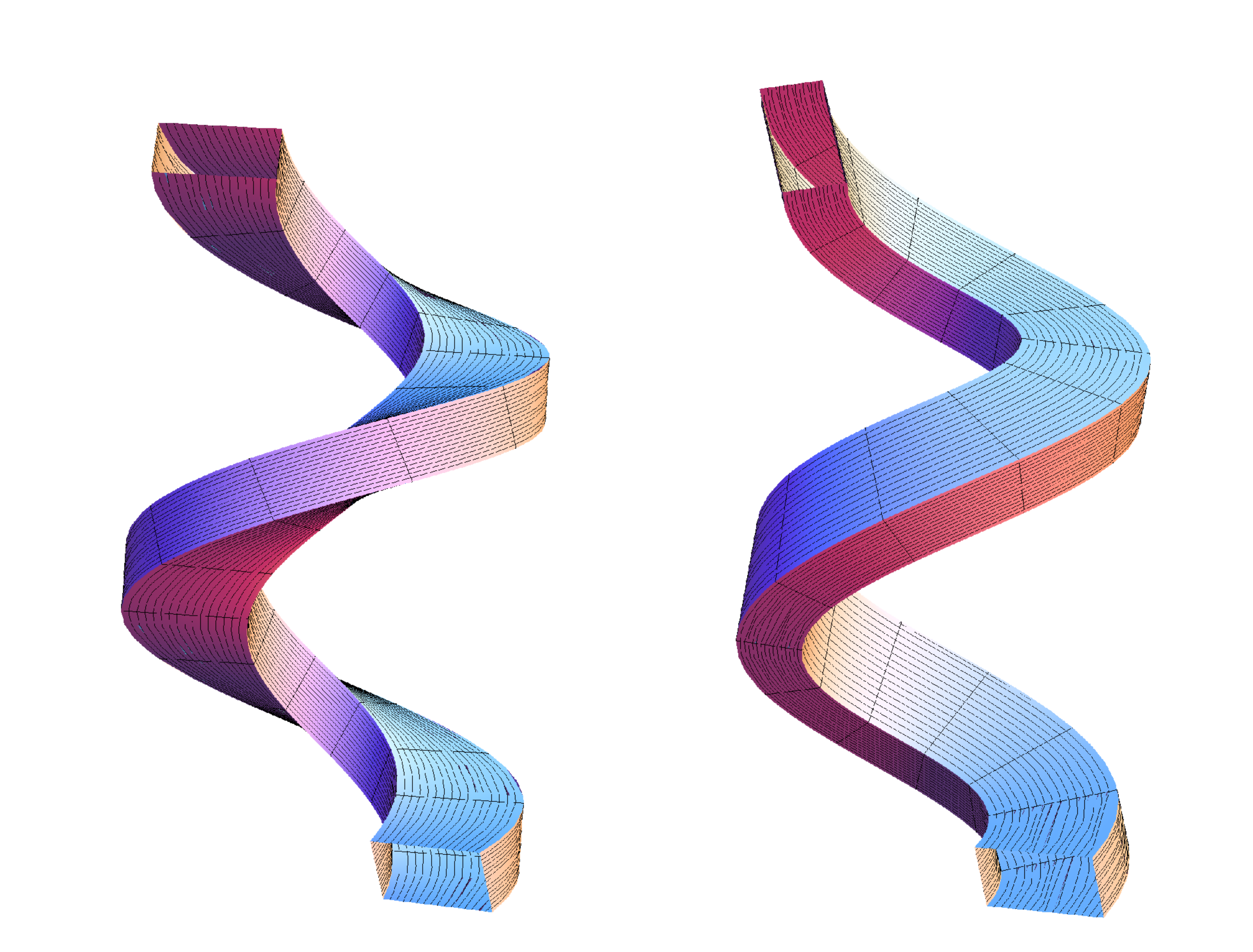}
\caption{A waveguide with rectangular cross-section built along a helix.
In the left figure the cross-section moves according
to the Frenet frame, \ie~$\dot{\theta}_F = 0$,
so that the waveguide is \emph{twisted}
because $\dot{\theta} = -\tau \not= 0$.
In the right figure the cross section moves according to the RPAF,
\ie~$\dot{\theta} = 0$;
we say that such a waveguide is \emph{untwisted}.}\label{frenetvsrpaf}
\end{center}
\end{figure}
\begin{remark}\label{remarknormres}
In Assumption~\ref{ass2} of Theorem~\ref{Thm.main.intro},
we state some requirements on the curvature functions $k_1,k_2$
that are not uniquely determined for the reference curve,
as we have seen in this subsection.
However, let us fix some particular RPAF
with curvatures $k_1^0,k_2^0$
and recall that the curvatures for different RPAFs
are only the linear combination of $k_1^0$ and $k_2^0$.
When we examine the condition~\eqref{sigmakknule},
we easily find that if $k_1^0,k_2^0$ satisfy it,
then all their linear combinations do satisfy it as well
(due to the triangle inequality in $L^2$).
Hence there is no ambiguity in Theorem~\ref{Thm.main.intro}.
\end{remark}
\subsection{The geometry of the tube}\label{sectube}
As mentioned in Introduction,
the tubes we consider in this paper are obtained
by translating and rotating a two-dimensional cross-section
along a spatial curve~$\Gamma$.
This definition can be formalized by means
of the RPAF $\{T,M_1,M_2\}$ found in the previous section.

The cross-section of our tube can be quite arbitrary.
We only assume that~$\omega$ is a bounded open connected subset of~$\R^2$.
The boundedness implies that the quantity
\begin{equation}\label{defa}
a := \sup_{t\in\omega}|t|
%<\infty
\end{equation}
is finite.
We say that~$\omega$ is \emph{circular}
if it is a disc or an annulus centered at the origin of~$\R^2$
(with the usual convention of identifying open sets
which differ on the set of zero capacity).

Given an angle function $\theta \in W_\mathrm{loc}^{1,\infty}(I)$,
let us define a rotation matrix
$$\mathcal{R}^{\theta} = \left(
\begin{array}{cc}
  \cos{\theta}    & \sin{\theta} \\
  -\sin{\theta} & \cos{\theta}   \\
\end{array}
\right)
.
$$
Then we define a general moving frame $\{M_1^{\theta}, M_2^{\theta}\}$
along~$\Gamma$ by rotating the RPAF $\{M_1, M_2\}$
by the angle~$\theta$, \ie,
$$
  M_{\mu}^{\theta} = \sum_{\nu=1}^2 \mathcal{R}^{\theta}_{\mu\nu}M_{\nu}
  \,, \qquad
  \mu=1,2
  \,.
$$

Let $\Omega_0:=I\times\omega$ be a straight tube.
We introduce a curved tube~$\Omega_\eps$
of uniform cross-section $\eps\omega$
as the image
\begin{equation}\label{tube.image}
  \Omega_\eps := \mathcal{L}(\Omega_0)
  \,,
\end{equation}
where the mapping
$\mathcal{L}:$ $\Omega_0\rightarrow \R^3$
is defined by
\begin{equation}\label{zavedeniL}
\mathcal{L}(s,t) := \Gamma(s) + \eps \sum_{\mu=1}^2 t_{\mu} M_{\mu}^{\theta}
\,.
\end{equation}

We say that the tube~$\Omega_\eps$ is \emph{bent}
if the reference curve~$\Gamma$ is not a straight line,
\ie~$\kappa\not=0$.
We say that~$\Omega_\eps$ is \emph{untwisted}
if~$\omega$ is circular or the cross-section is moved along~$\Gamma$
by a RPAF, \ie~$\dot\theta=0$;
otherwise the tube is said to be \emph{twisted} (for example of twisted and untwisted tube see Figure~\ref{frenetvsrpaf}).
A list of equivalent conditions for twisting
can be found in~\cite{K6-with-erratum}.

It is usual in the theory of quantum waveguides
to assume the tube~$\Omega_\eps$ is non-self-intersecting,
\ie, $\mathcal{L}$~is injective.
The necessary (but not always sufficient) condition
for the injectivity is the non-vanishing determinant of the metric tensor
$$
G_{ij}:=\partial_i \mathcal{L}\cdot \partial_j \mathcal{L}.
$$
Here $\partial_i$ denotes the partial derivative with respect
to the~$i^{th}$ variable,
where the ordered set ($s$, $t_1$, $t_2$) corresponds to (1,2,3).
Employing~\eqref{cartan},
it is straightforward to check that
the matrix $G=(G_{ij})$ reads
\begin{equation}\label{metrika}
G = \left(\begin{array}{ccc}
  h^2 + \eps^2(h_2^2 + h_3^2)    & -\eps^2 h_3 & -\eps^2 h_2\\
 -\eps^2 h_3 & \eps^2   & 0  \\
  -\eps^2 h_2 & 0   & \eps^2  \\
\end{array}
\right)
,
\end{equation}
where
\begin{equation}\label{hdef}
\begin{aligned}
h(\cdot,t)&:=1 - \eps \, t_1\left(k_1\cos{\theta} + k_2\sin{\theta}\right)
- \eps \, t_2\left(-k_1\sin{\theta}+k_2\cos{\theta}\right) , \\
h_2(\cdot,t)&:=-t_1\,\dot{\theta} , \\
h_3(\cdot,t)&:=t_2\,\dot{\theta} .
\end{aligned}
\end{equation}

We have
\begin{equation}\label{determinant}
|G| := \det(G) =\eps^4 \;\! h^2
\,,
\end{equation}
hence the condition on the determinant being everywhere nonzero
requires that~$h$ is a positive function.
%$$
%\eps t_1\left(k_1\cos{\theta} + k_2\sin{\theta}\right) + \eps t_2\left(-k_1\sin{\theta}+k_2\cos{\theta}\right) < 1 \qquad \forall(s,t)\in \Omega_0.
%$$
The latter can be satisfied only if the functions $k_1,k_2$ are bounded.
Therefore we always assume
\begin{equation}\label{Ass.basic}
  \kappa \in L^\infty(I)
  \,,
\end{equation}
which is equivalent to the boundedness of $k_1,k_2$ due to~\eqref{defkappa}.
In particular, we have
\begin{equation}\label{omezenostk}
\|k_{\mu}\|_\infty
\leq \|\kappa\|_{\infty}<\infty,
\qquad \mu=1,2,
\end{equation}
where
$
  \|\cdot\|_{\infty} := \|\cdot\|_{L^\infty(I)}
$.
%
%\begin{equation}\label{omezenostk}
%k_{\mu}(s)\leq \|\kappa\|_{\infty}<\infty \,,
%\qquad \forall s\in I, \quad \mu=1,2\,,
%\end{equation}
%
%where
%$
%  \|\cdot\|_{\infty} := \|\cdot\|_{L^\infty(I)}
%$.
Using in addition the boundedness of~$\omega$,
we find the bound
\begin{equation}\label{Ass.basic.pre}
  h(s,t) \geq 1 - \eps \;\! a \;\! \|\kappa\|_{\infty}
\end{equation}
for every $(s,t)\in\Omega_0$.
This ensures the positivity of~$h$
for all sufficiently small~$\eps$.

%Moreover, in the estimates we will often use that $\eps$ can be chosen so small that $16 \eps a \Ck \leq 1$ and thus
%\begin{equation}\label{omezh}
%\frac{3}{4}\leq 1-\eps a C_k \leq h \leq 1+\eps a C_k\leq \frac{5}{4}.
%\end{equation}

If the determinant~\eqref{determinant} is positive,
the matrix~\eqref{metrika} is invertible and we have
\begin{equation*}
  G^{-1} = \frac{1}{h^2}
  \begin{pmatrix}
    1 & h_3 & h_2 \\
    h_3 & \eps^{-2} h^2+h_3^2 & h_2 h_3 \\
    h_2 & h_3 h_2 & \eps^{-2} h^2+h_2^2 \\
  \end{pmatrix}
  .
\end{equation*}

Summing up, assuming~\eqref{Ass.basic} and the injectivity of~$\mathcal{L}$,
the mapping induces a global diffeomorphism
between the straight tube~$\Omega_0$ and~$\Omega_\eps$,
and the latter has the usual meaning of a non-self-intersecting
curved tube embedded in~$\R^3$.
For sufficient conditions ensuring the injectivity of~$\mathcal{L}$
we refer to \cite[App.~A]{EKK}.
The above construction gives rise to Assumption~\ref{assumgamma}.

\begin{remark}\label{Rem.self}
Abandoning the geometrical interpretation of~$\Omega_\eps$
being a non-self-intersecting tube in~$\R^3$,
it is possible to consider $(\Omega_0,G)$
as an abstract Riemannian manifold,
not necessarily embedded in~$\R^3$.
This makes~\eqref{Ass.basic}
(together with the smallness of~$\eps$
to ensure that the right hand side of~\eqref{Ass.basic.pre}
is positive) the only important hypothesis in the present study.
In other words, the injectivity assumption~(iii)
in Assumption~\ref{assumgamma} can be relaxed,
the results of the present paper hold
in this more general situation.
\end{remark}
\subsection{The Hamiltonian}\label{sechamil}
Let us now consider~$\Omega_\eps$ as the configuration space
of a quantum waveguide.
We assume that the motion of a quantum particle
inside the waveguide is effectively free
and that the particle wavefunction is suppressed
on the boundary of the tube.
Hence, setting $\hbar=2m=1$,
the one-particle Hamiltonian acts as the Laplacian on $L^2(\Omega_\eps)$
subjected to Dirichlet boundary conditions on~$\partial\Omega_\eps$:
\begin{equation}\label{puvodniham}
-\Delta^{\Omega_\eps}_D.
\end{equation}
The objective of this subsection is to give a precise
meaning to this operator.

The most straightforward way is to assume that~$\mathcal{L}$
is injective and define~\eqref{puvodniham}
as the \emph{Dirichlet Laplacian} on $L^2(\Omega_\eps)$.
Indeed, this is well defined for open sets
and from the previous subsection we know that
$\mathcal{L}$~induces a global diffeomorphism,
so that, in particular, $\Omega_\eps$~is open.
More specifically,
the Dirichlet Laplacian $-\Delta^{\Omega_\eps}_D$
is introduced as the self-adjoint operator associated on $L^2(\Omega_\eps)$
with the closed quadratic form
$$
  Q_D^{\Omega_\eps}[\psi]:=\|\nabla\psi\|^2_{L^2(\Omega_\eps)}
  \,,
  \qquad \Dom(Q_D^{\Omega_\eps}) = W^{1,2}_0(\Omega_\eps).
$$

From this point of view, we regard the tube as a submanifold of~$\R^3$.
For the description of~$\Omega_\eps$,
the most suitable coordinates are the curvilinear `coordinates'
$(s,t)\in\Omega_0$ defined via the mapping~$\mathcal{L}$ in~\eqref{zavedeniL}.
They are implemented by means of the unitary transform
\begin{equation}\label{unitary1}
  U_1 : L^2(\Omega_\varepsilon)
  \to \tilde{\H}_{\eps} := L^2\big(
  \Omega_0,|G(s,t)|^{1/2}\;\!ds\;\!dt
  \big):
  \big\{
  \psi \mapsto \psi \circ \mathcal{L}
  \big\}
\end{equation}
mentioned already in \eqref{U1}.
The transformed operator $\tilde{H}_\eps:=U_1(-\Delta^{\Omega_\eps}_D)U_1^{-1}$
can be determined as the operator associated with the transformed form
\begin{equation}\label{formalapbel}
\tilde{Q}_{\eps}[\psi]
:= Q_D^{\Omega_\eps}[U_1^{-1}\psi]
= \left(
\partial_i\psi,G^{ij}\partial_j\psi
\right)_{\tilde{\H}_{\eps}}
\,, \qquad
\Dom(\tilde{Q}_\eps) := U_1 W^{1,2}_0(\Omega_\eps)
\,.
\end{equation}
Here $G^{ij}$ are coefficients of the inverse metric~\eqref{metrika}
and the Einstein summation convention is adopted
(the range of indices~$i,j$ being $1,2,3$).
In a weak sense, $\tilde{H}_\eps$~acts as the Laplace-Beltrami operator
$
  -|G|^{-1/2}\partial_i|G|^{1/2} G^{ij}\partial_j
$,
but we shall not need this fact,
working exclusively with quadratic forms in this paper.

Let us emphasize that,
for the quadratic form~$\tilde{Q}_{\eps}$ to be well defined,
the matrix~$G$ does not need to be differentiable,
a local boundedness of its elements is sufficient.
As a matter of fact,
the form domain $U_1 W^{1,2}_0(\Omega_0)$ can be alternatively
characterized as the completion of $C_0^\infty(\Omega_0)$
with respect to the norm
$$
  \|\psi\|_{\tilde{Q}_{\eps}} := \sqrt{
  \left(
  \partial_i\psi,G^{ij}\partial_j\psi
  \right)_{\tilde{\H}_{\eps}}
  + \|\psi\|_{\tilde{\H}_{\eps}}^2
  }
  \,.
$$
If the functions~$\kappa$ and~$\dot\theta$ are bounded,
it is possible to check that the $\tilde{Q}_{\eps}$-norm
is equivalent to the usual norm in $W^{1,2}(\Omega_0)$.
For this reason, in addition to~\eqref{Ass.basic},
we assume henceforth the global boundedness
\begin{equation}\label{Ass.basic.bis}
  \dot\theta \in L^\infty(I)
  \,.
\end{equation}
Then we have
\begin{equation}\label{form.domain}
  \Dom(\tilde{Q}_\eps) = W^{1,2}_0(\Omega_0)
  \,.
\end{equation}
\begin{remark}\label{Rem.self.bis}
Now, if~$\mathcal{L}$ is not injective,
the image~\eqref{tube.image} might be quite complex
and the standard notion of the Dirichlet Laplacian
on~$\Omega_\eps$ meaningless.
Nevertheless, the Laplace-Beltrami operator
on the Riemannian manifold $(\Omega_0,G)$,
\ie\ the operator associated
on~$\tilde{\H}_\eps$ with the closure of the form
defined via the second identity in~\eqref{formalapbel}
on the initial domain $C_0^\infty(\Omega_0)$,
is fully meaningful.
Moreover, it coincides with~$\tilde{H}_\eps$ defined above.
This is the way how to transfer the results of the present paper
to the more general situation of Remark~\ref{Rem.self}.
In particular, Theorem~\ref{Thm.main.intro} holds
without the injectivity assumption~(iii)
in Assumption~\ref{assumgamma}
provided that we properly reinterpret
the meaning of~\eqref{puvodniham}
as the Laplace-Beltrami operator~$\tilde{H}_\eps$ in $(\Omega_0,G)$
and we write just~$U_2$ instead of~$U$
in the statement of the theorem
when dealing with the more general situation.
\end{remark}

Finally, let us recall that the spectrum of~$\tilde{H}_{\eps}$
explodes as $\eps^{-2} E_1$ in the limit as $\eps \to 0$.
This is related to the fact that the ground-state eigenvalue
of the cross-sectional Laplacian $-\Delta_D^{\eps\omega}$
equals $\eps^{-2} E_1$.
Therefore, to get a non-trivial limit,
we rather consider the renormalized operator
$$
  \tilde{\tilde{H}}_\eps := \tilde{H}_{\eps} - \eps^{-2} E_1
$$
in the sequel.

%----------------------------------%
\section{The mollification strategy}\label{Sec.Steklov}
%----------------------------------%
%
Our strategy how to reduce the regularity assumptions
about the waveguide
consists of the three points (I)--(III)
roughly mentioned in Section~\ref{Sec.strategy.intro}.
The first of them, \ie~the usage of RPAF instead of the Frenet frame,
was already explained in Section~\ref{secframe}.
The item~(II) consists in working with associated sesquilinear forms
instead of operators.
In the preceding Section~\ref{sechamil},
we introduced the Dirichlet Laplacian in the tube~$\Omega_\eps$
using exclusively quadratic forms
and it enables us to understand the derivatives
in the weak sense and to reduce the requirements
on the differentiability of the reference curve.

However, as explained in Section~\ref{Sec.strategy.intro},
for the standard operator procedure to work,
certain additional smoothness of curvature functions are still needed.
In this section we propose a method how to proceed
without any extra regularity hypotheses.
It is based on the mollification procedure~(III)
sketched in Section~\ref{Sec.strategy.intro}
and we believe it might be useful in other problems as well.

\subsection{The modified unitary transform}\label{secsmooth}
As explained in Section~\ref{Sec.strategy.intro},
the main idea consists in mollifying the curvatures $k_1,k_2$
by means of the Steklov approximation to get $k_1^\eps,k_2^\eps$
introduced in~\eqref{aproximace}.
If~$I$ is finite or semi-infinite, we adopt the convention
of Assumption~\ref{ass2} to give a meaning
to function values outside~$I$ in the definition.
That is, we assume that~$k_\mu$, with $\mu=1,2$,
are extended from~$I$ to~$\R$ by zero.

The definition~\eqref{aproximace}
involves a positive continuous function
$\eps \mapsto\delta(\eps)$
which is supposed to satisfy
\begin{equation}\label{deltaknule}
\lim_{\eps\rightarrow 0}\delta(\eps)= 0
\qquad\textrm{and}\qquad
\lim_{\eps\rightarrow 0}\frac{\eps}{\delta(\eps)}= 0.
\end{equation}
Here the first assumption is reasonable since then
$k_{\mu}^{\eps} \xrightarrow[]{\eps\rightarrow 0} k_{\mu}$
in a certain sense (see Section~\ref{secsteklov} below).
The relevance of the second condition will become clear
in our computations.

It follows from~\eqref{defkappa} that
\begin{equation}\label{omezenostkeps}
    \|k_{\mu}^{\eps}\|_\infty \leq \Ckappa,
    \qquad {\mu}=1,2.
\end{equation}
Furthermore, the mollified functions are differentiable
for any positive~$\eps$,
\begin{equation}\label{deraproximace}
\dot{k}^{\eps}_{\mu}(s)
=\frac{k_{\mu}(s+\frac{\delta(\eps)}{2})
-k_{\mu}(s-\frac{\delta(\eps)}{2})}{\delta(\eps)},
\qquad {\mu}=1,2
,
\end{equation}
although the derivative might diverge in the limit as $\eps \to 0$
for non-differentiable functions~$k_\mu$.

We also introduce a smoothed version
of the `Jacobian'~\eqref{hdef}
$$
  \heps(\cdot,t)
  := 1 - \eps \, t_1\left(\keps_1\cos{\theta} + \keps_2\sin{\theta}\right)
  - \eps \,  t_2\left(-\keps_1\sin{\theta}+\keps_2\cos{\theta}\right)
$$
and of the determinant~\eqref{determinant},
$
  |\tilde{G}|:= \eps^4 \heps^2
$.
The latter will be used to generate a modified version
of the standard unitary transform~$U_2$ from Section~\ref{Sec.strategy.intro}.
We define
\begin{equation}\label{U2.tilde}
  \tilde{U}_{2}: \tilde{\H}_{\eps} \to \H_{\eps} :=
  L^2\left(
  \Omega_0, \frac{|G(s,t)|^{1/2}}{|\tilde{G}(s,t)|^{1/2}}\,ds\,dt
  \right):
  \left\{ \psi \mapsto |\tilde{G}|^{1/4} \psi \right\}
  .
\end{equation}
The norm and inner product
in the Hilbert space~$\H_{\eps}$
will be denoted by $\|\cdot\|_{\eps}$
and $(\cdot,\cdot)_\eps$, respectively.

In addition to~\eqref{Ass.basic},
let us assume~\eqref{Ass.basic.bis} in the following,
so that the form domain of~$\tilde{H}$
is given by~\eqref{form.domain}.
Since the Sobolev space $W_0^{1,2}(\Omega_0)$ is left invariant
by the modified transform~$\tilde{U}_{2}$, the operator
\begin{equation}\label{defheps}
H_{\eps}:=\tilde{U}_{2}\tilde{\tilde{H}}_{\eps}\tilde{U}_{2}^{-1}
= |\tilde{G}|^{1/4}\tilde{\tilde{H}}_{\eps}|\tilde{G}|^{-1/4}
\end{equation}
is well defined in the form sense.
The associated quadratic form reads
\begin{align}
\label{qeps}Q_{\eps}[\psi]
=&
\intom\frac{1}{hh_{\eps}}
\left|\der\psi\right|^2 \,ds\,dt
+ \frac{1}{\eps^2}\intom \frac{h}{\heps}|\nabla'\psi|^2\,ds\,dt
- \frac{E_1}{\eps^2}\intom \frac{h}{\heps}|\psi|^2\,ds\,dt
\\
\nonumber&+\,\frac{1}{2}
\intom\frac{1}{\heps^2}(k_1\keps_1+k_2\keps_2)|\psi|^2\,ds\,dt
- \frac{3}{4}\intom\frac{h}{\heps^3}
\left((\keps_1)^2+(\keps_2)^2\right)|\psi|^2\,ds\,dt
\\
\nonumber&+\intom\frac{\left(\der\heps\right)^2}{4h\heps^3} |\psi|^2\,ds\,dt
- \intom\frac{\der\heps}{h\heps^2}\,\Re(\bar{\psi}\der\psi)\,ds\,dt,
\end{align}
with $\psi\in\Dom(Q_\eps)=W_0^{1,2}(\Omega_0)$.
Here $\nabla'$ is the gradient operator in
the `transverse' variables $(t_1,t_2)$
and~$\ader$ is the transverse angular-derivative operator
$$
  \ader
  := (t_2,-t_1)\cdot\nabla'
  = t_2\frac{\partial}{\partial t_1}-t_1\frac{\partial}{\partial t_2}
  \,.
$$

An important feature of the operator $H_{\eps}$ is
its boundedness from below.
We prove it together with another relation used in our computations below.
\begin{lemma}\label{lemmafiner}
Let $Q_{\eps}$ be the quadratic form defined by~\eqref{qeps}
and let~\eqref{deltaknule} be satisfied.
Then for all $\psi\in W^{1,2}_0(\Omega_0)$ and small enough $\eps$
\begin{equation}\label{qepsgeq}
Q_{\eps}[\psi] \geq \frac{1}{2}\intom\frac{1}{hh_{\eps}}
\left|\der\psi\right|^2 \,ds\,dt
- 9 \Ckappa^2 \|\psi\|^2_{\eps}.
\end{equation}
\end{lemma}
\begin{proof}
If we assume that $\eps$ is so small that
$$
\frac{3}{4}\leq 1-\eps a \Ckappa \leq h \leq 1+\eps a \Ckappa \leq \frac{5}{4}
\,,
$$
then the same relation holds for $h_{\eps}$
and using~\eqref{omezenostk} we easily get
$$
  \left|\frac{1}{2}\intom\frac{1}{\heps^2}(k_1\keps_1+k_2\keps_2)|\psi|^2\,ds\,dt
  - \frac{3}{4}\intom\frac{h}{\heps^3}\left((\keps_1)^2+(\keps_2)^2\right)
  |\psi|^2\,ds\,dt\right|\leq 5 \Ckappa^2\|\psi\|_{\eps}^2.
$$
The estimate on terms proportional to $\eps^{-2}$
is based on the Poincar\'e-type inequality
$$
  \int_{\omega} |\nabla'\phi|^2 dt - E_1\int_{\omega} |\phi|^2 dt \geq 0
$$
that holds for all $\phi\in W^{1,2}_0(\omega)$,
and on Fubini's theorem.
Due to nontrivial measure in our integrals,
we have to use substitution $\phi:=\sqrt{h/\heps}\;\!\psi$,
then we obtain
\begin{eqnarray*}
\lefteqn{
\frac{1}{\eps^2}\intom \frac{h}{\heps}|\nabla'\psi|^2\,ds\,dt
- \frac{E_1}{\eps^2}\intom\frac{h}{\heps}|\psi|^2\,ds\,dt}
\\
&&=\frac{1}{\eps^2}\intom\left(|\nabla'\phi|^2
- E_1|\phi|^2\right)\,ds\,dt
\\
&&\phantom{=} +\intom\left(\frac{3(k_1^2+k_2^2)}{4h^2}
-\frac{k_1\keps_1+k_2\keps_2}{2h\heps}
-\frac{(\keps_1)^2+(\keps_2)^2}{4\heps^2}\right)|\phi|^2\,ds\,dt
\\
&&\geq - 3 \Ckappa^2 \intom|\phi|^2\,ds\,dt
=  - 3 \Ckappa^2 \|\psi\|^2_{\eps}.
\end{eqnarray*}

The last term in (\ref{qeps}) can be estimated using
the Schwarz inequality and the simple Young's inequality ($2ab\leq a^2+b^2$)
\begin{multline}\label{odhad1}
\left|\intom\frac{\der\heps}{h\heps^2}\,\Re(\bar{\psi}\der\psi)\,ds\,dt\right|
\\
\nonumber \leq \frac{1}{2}
\left[\intom\frac{\left(\der\heps\right)^2}{4h\heps^3}|\psi|^2\,ds\,dt
+ \intom\frac{1}{hh_{\eps}}
\left|\der\psi\right|^2 \,ds\,dt\right].
\end{multline}
Summing up, we get
$$
Q_{\eps}[\psi] \geq \frac{1}{2}\intom\frac{1}{hh_{\eps}}\left|\der\psi\right|^2 \,ds\,dt - 8 \Ckappa^2 \|\psi\|^2_{\eps} - \intom\frac{\left(\der\heps\right)^2}{4h\heps^3} |\psi|^2\,ds\,dt
$$
and the proof is finished by the estimate
$$
\left|\intom\frac{\left(\der\heps\right)^2}{4h\heps^3}|\psi|^2\,ds\,dt \right|\leq 51 \frac{\eps^2}{\delta(\eps)^2} a^2 \Ckappa^2 \|\psi\|^2_{\eps} \leq \Ckappa^2 \|\psi\|^2_{\eps}
$$
which holds for small enough $\eps$ due to~\eqref{deltaknule}.
\end{proof}

As a consequence of Lemma~\ref{lemmafiner}, we get that
%
%\begin{equation}\label{boundqeps}
%Q_{\eps}[\psi] \geq - 9 \Ckappa^2 \|\psi\|^2_{\eps}
%\end{equation}
%
%then yields that
for a constant $\lambda < - 9 \Ckappa$
the operator $\Heps-\lambda$ is invertible and it holds
\begin{equation}\label{omezinvheps}
\|(H_{\eps}-\lambda)^{-1}\|_{\B(\Heps)}\leq\frac{1}{|\lambda|-9 \Ckappa^2}
\,.
\end{equation}
\subsection{Convergence properties of the Steklov approximation}\label{secsteklov}
Let~$f$ be a bounded function defined on an interval~$I$
and let $f^{\eps}$ be its Steklov approximation
\begin{equation}\label{steklovf}
f^{\eps}(s):=\frac{1}{\delta(\eps)}
\int_{s-\frac{\delta(\eps)}{2}}^{s+\frac{\delta(\eps)}{2}}f(\xi)\,d\xi
\end{equation}
where~$\delta$ is a positive continuous function on~$\R$ satisfying
the first of the requirements of~\eqref{deltaknule}.
Again, we recall the extension convention of Assumption~\ref{ass2}
if~$I$ is finite or semi-infinite.

In the computations below we require for $f=k_1, k_2, \dot{\theta}$
that in a certain sense $f^{\eps}$ converges to $f$
in the limit $\eps\to 0$. Namely, the crucial requirement reads
\begin{equation}\label{intf}
\left(\int_I\left|f-f^{\eps}\right|^2 |\varphi|^2 ds\right)^{1/2}
\xrightarrow{\eps\rightarrow 0} 0
\,, \qquad \forall \varphi\in W^{1,2}_0(I).
\end{equation}
In the following we will find the requirements on~$f$
such that this condition is satisfied.

We shall start with the estimate on the left hand side of~\eqref{intf}.
\begin{lemma}\label{lemmarozklad}
Let $f\in L^{\infty}(I)$ and let $f^{\eps}$ be the Steklov approximation of $f$.
Let $\varphi\in W^{1,2}(I)$ and finally let $\{a_n\}_{n=n_{-}}^{n_{+}}\subset I$,
$n\in \Z$ be the strictly increasing sequence of numbers
where $a_{n_-} = \inf_{s\in I} s$, $a_{n_+} = \sup_{s\in I} s$,
all the intervals $I_n:=(a_n,a_{n+1})$
are finite and $n_{\pm}$ can be either finite number or $\pm\infty$.
Then
\begin{equation}\label{rozkladint}
\int_I\left|f-f^{\eps}\right|^2 |\varphi|^2 ds\leq \sup_{n_-\leq n\leq n_+} \left[\left(\frac{\|f-f^{\eps}\|^2_{L^2(I_n)}}{a_{n+1}-a_n}\right) + 2 \|f-f^{\eps}\|^2_{L^2(I_n)}\right]\|\varphi\|^2_{W^{1,2}(I)}.
\end{equation}
\end{lemma}
\begin{proof}
The main idea of the proof is rewriting the estimated expression as
$$
\int_I|f-f^{\eps}|^2 |\varphi|^2 ds = \sum_{n=n_-}^{n_+-1}\int_{a_n}^{a_{n+1}} \dot{g}_{\eps}^n |\varphi|^2 ds = \sum_{n=n_-}^{n_+-1}\left(\left[{g}_{\eps}^n|\varphi|^2\right]_{a_n}^{a_{n+1}} - \int_{a_n}^{a_{n+1}}g_{\eps}^n\, 2\Re\left(\bar{\varphi}\dot{\varphi}\right) ds\right)
$$
where we integrated by parts and where we defined, for all $s\in I$,
$$g_{\eps}^n(s):=\int_{a_n}^s|f(\xi)-f^{\eps}(\xi)|^2\chi_n(\xi)d\xi,$$
$\chi_n(\xi)$ is the characteristic function of the interval $I_n$
(\ie, $\chi_n(\xi)=1$ for $\xi\in[a_n,a_{n+1}]$
and $\chi_n(\xi)=0$ elsewhere).
The proof is then completed by using the Schwarz and Young inequalities
and the relations
\begin{align*}
\sup_{s\in I} g_{\eps}^n(s) &= g_{\eps}^n(a_{n+1})
=  \int_{a_n}^{a_{n+1}}|f(\xi)-f^{\eps}(\xi)|^2 d\xi
= \|f-f^{\eps}\|^2_{L^2(I_n)},\\
|\varphi(a_{n+1})|^2
&\leq \left(\frac{1}{a_{n+1}-a_n}+1\right)\|\varphi\|_{W^{1,2}(I_n)}.
\end{align*}
\end{proof}

The lemma is of great importance for our computations.
From the generalized Minkowski inequality it follows that
\begin{equation}\label{minkowski}
\|f-f^{\eps}\|_{L^2(I_n)}
\leq \sup_{|\eta|\leq \frac{\delta(\eps)}{2}}
\left(\int_{I_n}\left|f(s)-f(s+\eta)\right|^2 ds\right)^{1/2}
=: \omega_2(\delta(\eps), f, I_n).
\end{equation}
Here the notational symbol~$\omega_p$ is adopted from \cite{Akhiezer},
where it is referred to as
\emph{modulus of continuity generalized to space $L^p$}
and is computed for a function $f$,
positive number $\delta(\eps)$ and interval $I_n$.
In~\cite{Akhiezer} it is also shown
that this quantity tends to zero with $\delta(\eps)$
if the interval $I_n$ is finite.
\footnote{More precisely, the proof in~\cite{Akhiezer}
is made for $\omega_2(\delta(\eps), f, \R)=:\omega_2(\delta(\eps), f)$
where $f\in L^2(\R)$.
Here we consider a bounded function defined on finite~$I_n$
and prolonged on $\R\setminus I_n$ in such a way that
the extension is an $L^2$-function.
Then the proof from~\cite{Akhiezer} can be used.}
This directly yields that if the interval $I$ is finite,
the convergence~\eqref{intf} holds.

If $I$ is infinite, we can cut it into finite intervals $(a_n,a_{n+1})$
and for each~$n$ the expression in square brackets
on the right hand side of~\eqref{rozkladint}
would tend to zero when $\eps\rightarrow 0$.
Unfortunately, the supremum over $n_-\leq n\leq n_+$
might not have the zero limit.
On the other hand, this may happen only in a case of functions
that oscillate quickly at infinity (see Example~\ref{excounter}).
In other words, also for unbounded~$I$, the condition~\eqref{intf}
is satisfied for functions that behave `reasonably' at infinity.

We summarize the above ideas in the following proposition.
For simplicity, we use the result of Lemma~\ref{lemmarozklad}
for the equidistant division $a_n := n$ and we recall the extension convention
of Assumption~\ref{ass2} for finite or semi-infinite~$I$.
\begin{proposition}\label{theorconv}
Let $\varphi\in W^{1,2}_0(I)$, $f\in L^{\infty}(I)$
and let $f^{\eps}$ be the Steklov approximation of~$f$
given by~\eqref{steklovf}. Then
\begin{equation}
\label{nadruhou}
\left(\int_I\left|f-f^{\eps}\right|^2 |\varphi|^2 ds \right)^{1/2}
\leq \sqrt{3} \ \sigma_f(\delta(\eps)) \, \|\varphi\|_{W^{1,2}(I)}
\end{equation}
where $\sigma_f(\delta(\eps))$ is given by~\eqref{sigmaf}.
Furthermore, the quantity $\sigma_f(\delta(\eps))$
tends to zero as $\eps\to 0$ if any of the following
conditions is satisfied:
\begin{itemize}
\item[\emph{(i)}] $I$ is finite,
\item[\emph{(ii)}] $f$ periodic on $I_\mathrm{ext}$,
\item[\emph{(iii)}]
$f\in L^2(I_\mathrm{ext})$,
\item[\emph{(iv)}]
$f\in C^0(\overline{I_\mathrm{ext}})$,
\end{itemize}
where $I_\mathrm{ext} := I\setminus \overline{I_\mathrm{int}}$
for some finite open (possibly empty) interval~$I_\mathrm{int}$.
\end{proposition}
\begin{proof}
The inequality~\eqref{nadruhou} follows from Lemma~\ref{lemmarozklad}
and the relation~\eqref{minkowski}.
It reminds to prove the convergence properties of~\eqref{sigmaf}.

Any bounded interval $I$ can be covered by intervals $I_n = (n,n+1]$,
$n\in \mathcal{I}$ with $\mathcal{I}\subset\Z$ finite.
Then $\sigma_f(\delta(\eps))$ is proportional to
\emph{maximum} of $\omega_2(\delta(\eps), f, I_n)$ over $n\in\mathcal{I}$.
Since $\omega_2(\delta(\eps), f, I_n)$ converges to zero
for every $n\in\mathcal{I}$ as we explained above,
$\sigma_f(\delta(\eps))$ converges to zero if interval $I$ is finite.
For the same reason the convergence of $\sigma_f(\delta(\eps))$
does not depend on the behaviour of~$f$ on a bounded subinterval
$I_\mathrm{int}\subset I$ in case of infinite or semi-infinite~$I$ .

Also the convergence of $\sigma_f(\delta(\eps))$ for periodic functions
follows from the fact that the period of length $q<\infty$
can be covered by finite number of intervals $(n,n+1]$.
(More straightforwardly, using the sequence $a_n = nq$
in Lemma~\ref{lemmarozklad},
we get $\sigma_f(\delta(\eps)) \propto \omega_2(\delta(\eps), f, (0,q))$,
which tends to zero as $\eps \to 0$.)

In the case of $L^2$-functions,
we use the fact that for any $\eps>0$ there exists
a finite interval $I_\mathrm{int}$ such that $f(s)< \eps$
for all $s\in \R\setminus I_\mathrm{int}$.
Then again the significant contribution to $\sigma_f(\delta(\eps))$
comes from the finite interval.

Finally, for the uniformly continuous functions
the situation is even more simple. Here already the quantity
$$
  \omega_\infty(\delta(\eps),f)
  := \sup_{|\xi_1-\xi_2|\leq \delta(\eps)}
  \left|f(\xi_1)-f(\xi_2)\right|\qquad \xi_1,\xi_2\in I_\mathrm{ext},
$$
called the \emph{modulus of continuity} in~\cite{Akhiezer},
tends to zero when $\delta(\eps)$ tends to zero
and $\sigma_f(\delta(\eps))$ can be estimated
by $\omega_\infty(\delta(\eps),f)$.
Let us note that if $I$ is bounded or semi-bounded,
the function need not to be uniformly continuous after
extension by zero outside $I$.
However, the convergence can be still ensured by dividing $I$
on the $\delta(\eps)$-neighbourhood of the end point(s) and the rest of $I$.
Then we can estimate the integral over inner part
by $\omega_\infty(\delta(\eps),f)$ since here the function
is indeed uniformly continuous, the remaining integral
can be estimated by $\|f\|_{\infty}\delta(\eps)$ which tends to zero as well.
\end{proof}
%

%--------------------------------------%
\section{The norm-resolvent convergence}\label{Sec.proof}
%--------------------------------------%
%
In this long and technically demanding section
we give a proof of Theorem~\ref{Thm.main.intro}.

\subsection{Comparing operators acting on different
Hilbert spaces}\label{seccompare}
We start with describing a way
how to understand the resolvent convergence of operators
$-\Delta_D^{\Omega_\eps}$ and~$H_\mathrm{eff}$ acting on different
Hilbert spaces $L^2(\Omega_\eps)$ and $L^2(I)$, respectively.

First of all, we recall that, in Section~\ref{secsmooth},
we introduced the operator~$H_{\eps}$ on $\H_{\eps}$
which is unitarily equivalent to
$-\Delta_D^{\Omega_\eps}-\eps^{-2}E_1$.
It is therefore enough to explain the resolvent convergence
of $H_{\eps}$ and~$H_\mathrm{eff}$.
Our strategy is to reconsider these operators
as certain operators on the fixed Hilbert space
$$
  \H_0 := L^2(\Omega_0)
$$
and to show that the error due to the replacement
becomes negligible in the limit as $\eps \to 0$.

Recall that we have denoted the norm and inner product
in the $\eps$-dependent Hilbert space~$\H_{\eps}$
by $\|\cdot\|_{\eps}$ and $(\cdot,\cdot)_{\eps}$, respectively.
We simply write $\|\cdot\|$ and $(\cdot,\cdot)$
for the norm and inner product in~$\H_0$.
Finally, $\|\cdot\|_I$ and $(\cdot,\cdot)_I$
stand for the norm and inner product in $L^2(I)$.

In order to have a way to compare operators
acting on~$\H_{\eps}$ and~$\H_0$,
let us introduce yet another unitary transform
\begin{equation}\label{defueps}
\Ueps : \H_{\eps} \to \H_0 :
  \left\{ \psi \mapsto \frac{|G|^{1/4}}{|\tilde{G}|^{1/4}} \psi \right\}
  \,.
\end{equation}
For the convenience of the reader, we present here the following
diagram explaining the relation with the other unitary transforms
introduced so far:
\begin{equation}\label{diagram}
  L^2(\Omega_\eps) \xrightarrow[]{\quad U_1 \quad}
\begin{aligned}
\xymatrix{
   & \H_0
   \\
   \tilde{\H}_\eps \ar[ur]^{U_2} \ar[dr]_{\tilde{U}_2} &
   \\
   & \H_\eps \ar[uu]_{U_\eps}
}
\end{aligned}
\end{equation}

It is important to emphasize that while the transformed resolvent
$\Ueps(\Heps - i)^{-1}\Ueps^{-1}$ on~$\H_0$ is well defined
(as a unitary transform of a bounded operator),
the similar expression for the (unbounded) operator~$H_\eps$
may not have any sense. Indeed, $\Heps$~acts as a differential operator,
while~$|G|$ may not be differentiable under our minimal assumption.
The same remark applies to~$\tilde{H}_\eps$ and~$U_2$,
as already pointed out in Section~\ref{Sec.strategy.intro}.

Summing up, using the unitary transforms described above,
it is possible to reconsider the resolvent of
the Dirichlet Laplacian $-\Delta_D^{\Omega_\eps}$
as an operator on~$\H_0$. It remains to explain
how to reconsider~$H_\mathrm{eff}$ acting on~$L^2(I)$
as an operator on the `larger' space~$\H_0$.
This is done by introducing the following subspace of~$\H_0$:
\begin{equation}\label{subspaceH0}
\H_0^1:=\left\{\psi\in\H_0\,\left|\right.\,\exists\varphi\in L^2(I),
\,\psi(s,t) = \varphi(s)\mathcal{J}_1(t) \right\}.
\end{equation}
Recall that~$\mathcal{J}_1$ denotes the eigenfunction corresponding
to first eigenvalue of the transverse Dirichlet Laplacian $-\Delta_D^\omega$;
we choose it to be positive and normalized to one in $L^2(\omega)$.
$\H_0^1$ is closed, hence
\begin{equation}\label{hildecomp}
\H_0=\H_0^1 \oplus (\H_0^1)^{\bot}
\end{equation}
and every function $\psi\in \H_0$ can be uniquely written as
\begin{equation}\label{defdecomp}
\psi=P_1\psi + (1-P_1)\psi =: \psi_1\mathcal{J}_1 +\psi^{\bot}
\end{equation}
with $\psi_1\mathcal{J}_1\in \H_0^1$, $\psi^{\bot}\in (\H_0^1)^{\bot}$ and $P_1$ being projection on $\H_0^1$,
\begin{equation}\label{P1}
(P_1\psi)(s,t)
:=\left(\int_{\omega}\mathcal{J}_1(t)\psi(s,t) dt\right)\mathcal{J}_1(t)
\equiv \psi_1(s)\mathcal{J}_1(t).
\end{equation}
To shorten the notation, we denote by $\psi_1\mathcal{J}_1$
the function $\psi_1\otimes\mathcal{J}_1$,
\ie\ the function on $I\times\omega$ which assumes values
as $\psi_1(s)\mathcal{J}_1(t)$.
Such a decomposition of functions $\psi\in\H_0$ will be extensively used
throughout the text with the same notation.

Now we can introduce the isometric isomorphism
$$
  \pi: \H_0^1 \to L^2(I) :
  \left\{
  \psi_1(s)\mathcal{J}_1(t) \mapsto \psi_1(s)
  \right\}
  \,.
$$
Let $\qeff$ be the quadratic form associated with the operator $\Heff$, \ie,
$$
  \qeff[\varphi]=\int_I |\dot\varphi(s)|^2 ds
  +C_\omega\int_I \dot{\theta}(s)^2 |\varphi(s)|^2 ds
  - \frac{1}{4}\int_I\kappa(s)^2|\varphi(s)|^2 ds
  \,, \qquad
  \Dom(\qeff) = W_0^{1,2}(I)
  \,.
$$
(Recall that the basic Assumption~\ref{assumgamma} requires
that both~$\kappa$ and~$\dot\theta$ are bounded functions,
so that~$\Heff$ is well defined as a bounded perturbation
of the one-dimensional Dirichlet Laplacian $-\Delta_D^I$.)
The form~$\qeff$ can be identified with the quadratic form~$\Qeff$
acting on the subspace~$\H_0^1$ as
\begin{equation}\label{defQeff}
\begin{aligned}
\Qeff[\psi_1\mathcal{J}_1]
&:=  \intom |\partial_s\psi_1\mathcal{J}_1|^2\,ds\,dt
-\frac{1}{4} \intom \kappa^2|\psi_1\mathcal{J}_1|^2\,ds\,dt
+ C_\omega\intom \dot{\theta}^2 |\psi_1\mathcal{J}_1|^2\,ds\,dt
= \qeff[\psi_1]
%= \qeff[\pi(\psi_1\mathcal{J}_1)]
\,,
\\
\Dom(\Qeff) &:= \{\psi\in\H_0^1 \, | \, \psi_1 \in W_0^{1,2}(I)\}
\,.
\end{aligned}
\end{equation}
In a similar way we can identify operators acting
on $\H_0^1 \subset \H_0$ and $L^2(I)$.
We tacitly employ the identification,
without writing down the identification mapping~$\pi$ explicitly
in the formulae.
In particular, denoting by~$0^{\bot}$
the zero operator on $(\H_0^1)^{\bot}$,
the operator $\left(\Heff-i\right)^{-1}\oplus 0^{\bot}$
can be understood as an operator acting on the whole space~$\H_0$.

\subsection{Proof of Theorem~\ref{Thm.main.intro}}\label{secproofth}
At first let us explain the connection between the operator $U(-\Delta_D^{\Omega_\varepsilon} -\varepsilon^{-2} E_1-i)^{-1}U^{-1}$ from formula~\eqref{NR.bound} and the operator $\Ueps(\Heps - i)^{-1}\Ueps^{-1}$ we spoke about in the previous section. More precisely, we show that these two operators are identical.
Indeed, recall that $U=U_2U_1$ where $U_1$ and $U_2$
are unitary transforms described in Section~\ref{Sec.strategy.intro} and in addition that $H_{\eps} = \tilde{U}_2 U_1(-\Delta_D^{\Omega_\varepsilon} -\varepsilon^{-2} E_1) U_1^{-1}\tilde{U}_2^{-1}$. Using the diagram~\eqref{diagram} we easily get that
$$U_2U_1(-\Delta_D^{\Omega_\varepsilon} -\varepsilon^{-2} E_1-i)^{-1}U_1^{-1}U_2^{-1} = U_{\eps}\tilde{U}_2U_1(-\Delta_D^{\Omega_\varepsilon} -\varepsilon^{-2} E_1-i)^{-1}U_1^{-1}\tilde{U}_2^{-1}U_{\eps}^{-1} = \Ueps(\Heps-i)^{-1}\Ueps^{-1}$$
and in following we will prove Theorem~\ref{Thm.main.intro} using the last expression.

Another point is that according to \cite{kato} (Theorem IV.2.25), if the formula~\eqref{NR.bound} is satisfied for a $\lambda$ from the resolvent set of $\Heff$, then it holds true for all such $\lambda$. In particular there exists a constant $C_{\lambda}$ such that
\begin{multline}\label{estlambda}
\left\|\Ueps(\Heps-i)^{-1}\Ueps^{-1}
- \left(\left(\Heff-i\right)^{-1}\oplus 0^{\bot}\right) \right\|_{\B(\H_0)}
\\
\leq C_{\lambda} \left\|\Ueps(\Heps-\lambda)^{-1}\Ueps^{-1}
- \left(\left(\Heff-\lambda \right)^{-1}\oplus 0^{\bot}\right) \right\|_{\B(\H_0)}.
\end{multline}
Hence our aim is to prove that the right hand side
of the last expression tends to zero for some
$\lambda<-9\Ckappa^2$, since such $\lambda$
belongs to resolvent set of $\Heff$
and also of $\Heps$ (\cf~\eqref{qepsgeq}
which yields $Q_{\eps}[\psi] \geq - 9 \Ckappa^2 \|\psi\|^2_{\eps}$,
similarly $\qeff[\psi] \geq -\frac{1}{4}\Ckappa^2\|\psi\|_{I}^2$).
This proof is divided into proof of two auxiliary lemmata,
where in every lemma we compare one of the operators
on the right hand side of~\eqref{estlambda} with the resolvent of operator
\begin{equation}\label{defh0}
H_0:=1\otimes\left(-\frac{1}{\eps^2}\Delta^{\omega}_D - \frac{E_1}{\eps^2}\right) + \left(-\Delta^I_D - \frac{\kappa^2}{4} + C_\omega \dot{\theta}^2\right)\otimes1.
\end{equation}
The crucial step lies in comparison of $H_0$ and $\Heps$ stated in the first of these lemmata.
\begin{lemma}\label{lemma1}
Let $\lambda<-9 \Ckappa^2$ be a real constant and let the assumptions of Theorem~\ref{Thm.main.intro} be satisfied. Then
\begin{equation}\label{normresolv}
\|U_{\eps}(H_{\eps}-\lambda)^{-1}\Ueps^{-1} - (H_0-\lambda)^{-1}\|_{\B(\H_0)}\leq \tilde{C}\left(\eps + \eps \left(\|\dot{k}^{\eps}_1\|_{\infty} + \|\dot{k}^{\eps}_2\|_{\infty}\right) + \sigma(\delta(\eps))\right)
\end{equation}
for some constant $\tilde{C}$.
\end{lemma}

Second lemma giving the comparison of $H_0$ and $\Heff$ represents only a tiny improvement of the result above.
\begin{lemma}\label{lemmanormresh0}
Let $H_0$ be the operator defined by~\eqref{defh0} and let $\Heff$ be the effective Hamiltonian~\eqref{defheff}. Then
$$\|\left(H_0-\lambda\right)^{-1} - \left(\left(\Heff-\lambda\right)^{-1}\oplus 0^{\bot}\right) \|_{\B(\H_0)}\leq \tilde{\tilde{C}} \eps$$
for some real constants $\tilde{\tilde{C}}$ and $\lambda< -9 \Ckappa^2$.
\end{lemma}

Proofs of Lemmata~\ref{lemma1} and~\ref{lemmanormresh0}
will be given in Sections~\ref{prooflemma1}
and~\ref{prooflemmanorm}, respectively.
These lemmata will be proved using
a trick employed originally in \cite{Friedlander-Solomyak_2007},
where the estimate on the norm of the difference of resolvents
is obtained by a usage of the associated quadratic forms.
We state the following proposition to give the reader basic
idea of this trick; in our proofs we have to modify it
due to distinctness of Hilbert spaces our operators act on
and the main idea could be hidden by loads of technicalities.

\begin{proposition}\label{proptrik}
Let $A, B$ be positive self-adjoint operators acting on Hilbert space $\H$ with $A^{-1},B^{-1}\in\B(\H)$ and let $Q_A$, $Q_B$ be associated sesquilinear forms with $\Dom Q_A = \Dom Q_B$. Let us assume that for all $\phi,\psi \in \Dom Q_A$
\begin{equation}\label{estQ}
\left|Q_A(\phi,\psi)-Q_B(\phi,\psi)\right|\leq \sigma \sqrt{Q_A[\phi]}\sqrt{Q_B[\psi]}.
\end{equation}
Then
$$\|A^{-1}-B^{-1}\|_{\B(\H)}\leq \sigma \sqrt{\|A^{-1}\|_{\B(\H)}}\sqrt{\|B^{-1}\|_{\B(\H)}}.$$
\end{proposition}
\begin{proof}
Due to the assumption~\eqref{estQ} it holds that for all $f,g\in\H$
\begin{align}
\nonumber |\left(f,(A^{-1}-B^{-1})g\right)|  &= |\left(A\phi,(A^{-1}-B^{-1})B\psi\right)| =  \left|Q_B(\phi,\psi)-Q_A(\phi,\psi)\right|\leq \\
\label{trik}&\leq \sigma \sqrt{Q_A[\phi]}\sqrt{Q_B[\psi]} \leq \sigma \sqrt{\|A^{-1}\|_{\B(\H)}}\sqrt{\|B^{-1}\|_{\B(\H)}}\|f\| \|g\|
\end{align}
where the choice $f=A\phi$, $g=B\psi$ is possible for all $f,g\in\H$ due to boundedness of $A^{-1}$ and $B^{-1}$. This choice ensures that $\phi\in\Dom Q_A \cap \Dom A$ which due to the representation theorem (see \cite{kato}) yields $(A\phi,\psi) = Q_A(\phi,\psi)$ for all $\psi\in \Dom Q_A$. Similarly $\psi\in \Dom Q_B \cap \Dom B$ and $(\phi,B\psi) = Q_B(\phi,\psi)$ for all $\phi\in \Dom Q_B$. Inequality~\eqref{trik} yields directly the statement of the Proposition.
\end{proof}

\subsection{Proof of Lemma~\ref{lemmanormresh0}}\label{prooflemmanorm}
The quadratic form associated with the operator $H_0$ reads
\begin{align}
\label{defq0}Q_0[\psi]
=&\intom|\partial_s\psi|^2\,ds\,dt+\frac{1}{\eps^2}\intom|\nabla'\psi|^2\,ds\,dt
- \frac{E_1}{\eps^2}\intom |\psi|^2\,ds\,dt
\\
\nonumber & + C_\omega\intom \dot{\theta}^2 |\psi|^2 \,ds\,dt
- \frac{1}{4}\intom\kappa^2|\psi|^2\,ds\,dt
\end{align}
with
$$\Dom{Q_0} = W^{1,2}_0(\Omega_0).$$
Due to the equality $-\Delta_D^{\omega}\mathcal{J}_1 = E_1\mathcal{J}_1$, this form acts on the set $W^{1,2}_0(\Omega_0)\cap\H_0^1$ in the same way as $\Qeff$ given by~\eqref{defQeff} which we identify with the quadratic form $\qeff$ acting on $W^{1,2}_0(I)$. This yields
\begin{equation}\label{qnula}
\qeff(\phi_1,\psi_1) - Q_0(\phi_1\mathcal{J}_1,\psi_1\mathcal{J}_1) = 0
\,,
\end{equation}
where we use the notation from~\eqref{defdecomp}
and which leads us to slightly modified choice of functions $f,g$
comparing to the proof of Proposition~\ref{proptrik}.
We assign $H_0^{\lambda}:= H_0-\lambda$,
$\Heff^{\lambda}:= \Heff-\lambda$
and we choose $f := H_0^{\lambda}\phi$,
$P_1 g := (\Heff^{\lambda}\psi_1)\mathcal{J}_1$,
with $g^{\bot},\psi^{\bot}$ unspecified.
If we denote by $Q_0^{\lambda}$ and $\qeff^{\lambda}$,
the quadratic forms associated to $H_0^{\lambda}$ and $\Heff^{\lambda}$,
respectively,
we can rewrite the term analogous to the one estimated in~\eqref{trik} as
\begin{multline*}
\left(H_0^{\lambda}\phi, \left[\left(H_0^{\lambda}\right)^{-1}
-\left(\left(\Heff^{\lambda}\right)^{-1}\oplus 0^{\bot}\right)\right]
\left((\Heff^{\lambda}\psi_1)\mathcal{J}_1 + g^{\bot} \right)\right)
\\
=\qeff^{\lambda}(\phi_1,\psi_1)
+ \left(\phi^{\bot},(\Heff^{\lambda}\psi_1)\mathcal{J}_1\right)
+ \left(\phi,g^{\bot}\right)
- Q_0^{\lambda}(\phi_1\mathcal{J}_1,\psi_1\mathcal{J}_1)
- Q_0^{\lambda}\left(\phi^{\bot},\psi_1\mathcal{J}_1\right)
= \left(\phi,g^{\bot}\right).
\end{multline*}
Here all the terms except of $\left(\phi,g^{\bot}\right)$ vanish
due to~\eqref{qnula} or due to the orthogonality of
$\psi^{\bot},\phi^{\bot}$ and $\mathcal{J}_1$.
Hence we can estimate
\begin{align}
\nonumber\left|\left(f, \left[\left(H_0^{\lambda}\right)^{-1}
- \left(\left(\Heff^{\lambda}\right)^{-1}
\oplus 0^{\bot}\right)\right]g\right)\right|
=\left|\left(\phi,g^{\bot}\right)\right|
&= \left|\left((H_0^{\lambda})^{-1}f,(1-P_1)g\right)\right|
\nonumber
\\
& \leq \|f\|\|g\| \|(H_0^{\lambda})^{-1}(1-P_1)\|_{\B(\H_0)}
\nonumber
\\
&\leq \eps C_{\bot} \sqrt{\|(H_0^{\lambda})^{-1}\|}\|f\|\|g\|
\,,
\label{relnormres}
\end{align}
where the last estimate follows from the relation
$$\|\psi^{\bot}\|\leq \eps C_{\bot} \sqrt{Q_0^{\lambda}[\psi]}$$
that will be proved in Section~\ref{proofrel}. The proof is completed using the estimate
\begin{equation}\label{omezinvh0}
\|(H_0^{\lambda})^{-1}\|_{\B(\H_0)}\leq \frac{1}{|\lambda|-\frac{\Ckappa^2}{4}}
\end{equation}
analogous to~\eqref{omezinvheps},
\ie with $\tilde{\tilde{C}} := C_{\bot}\sqrt{\frac{1}{|\lambda|-\frac{\Ckappa^2}{4}}}$.

\subsection{Proof of Lemma~\ref{lemma1}}\label{prooflemma1}
In this proof the crucial and most tedious point is to check
that the assumption~\eqref{estQ} of the Proposition~\ref{proptrik} holds true.
Let $Q_{\eps}^{\lambda}[\psi] := \Qeps[\psi] - \lambda \|\psi\|_{\eps}$
be the quadratic form associated with the operator
$\Heps-\lambda =: \Heps^{\lambda}$.
Recall that $Q_0^{\lambda}$ was introduced in previous section.
If we assume that these forms (understood as sesquilinear forms) satisfy
\begin{equation}\label{estq2}
\left|\Qeps^{\lambda}(\phi,\psi) - Q_0^{\lambda}(\phi,\psi)\right|\leq \tilde{\sigma}(\eps) \sqrt{Q_0^{\lambda}[\phi]\Qeps^{\lambda}[\psi]}
\end{equation}
for all $\phi,\psi \in W_0^{1,2}(\Omega_0)$,
then we can derive the statement of Lemma~\ref{lemma1}
using similar ideas as in Proposition~\ref{proptrik}.
We only have to realize that the operators $\Heps$ and $H_0$
act on different Hilbert spaces,
in fact $\Heps$ is not self-adjoint on $\H_0$ where $H_0$ acts.
However, the Hilbert spaces $\H_{\eps}$ and $\H_0$
can be identified via the unitary transform~$\Ueps$
defined in~\eqref{defueps},
%\eg for scalar products on $\H_{\eps}$ and $\H_0$
%it holds $(\phi,\psi)_{\eps} = (\Ueps\phi,\Ueps\psi)$
which leads to the estimate
\begin{align*}
\left|\left(f, \left[U_{\eps}(H_{\eps}^{\lambda})^{-1}
\Ueps^{-1}-(H_0^{\lambda})^{-1}\right]g\right)\right|
\leq \left|Q_0^{\lambda}\left((H^{\lambda}_0)^{-1}f,(\Heps^{\lambda})^{-1}g\right)
- \Qeps^{\lambda}
\left((H_0^{\lambda})^{-1}f,(\Heps^{\lambda})^{-1}g\right)\right| \,&
\\
+ \left|\left(f,(U_{\eps}-1)(H_{\eps}^{\lambda})^{-1}\Ueps^{-1}g\right)\right| + \left|\left(f,(H_{\eps}^{\lambda})^{-1}(U_{\eps}^{-1}-1)g\right)\right| + \left|\left((H_0^{\lambda})^{-1}f,(\Ueps^2-1) g\right)\right|.&
\end{align*}
The operator $\Ueps$ differs from identity only by amount proportional to $\eps$, which we use in the estimate of last 3 terms. Together with \eqref{estq2} the final estimate reads
\begin{equation}\label{castlemma1}
\left|\left(f, \left[U_{\eps}(H_{\eps}^{\lambda})^{-1}\Ueps^{-1}-(H_0^{\lambda})^{-1}\right]g\right)\right|\leq \left(c\eps + 2\tilde{\sigma}(\eps) \sqrt{\|(H_0^{\lambda})^{-1}\|_{\B(\H_0)} \|(\Heps^{\lambda})^{-1}\|_{\B(\H_{\eps})}} \right) \|f\| \|g\|
\end{equation}
where $c = 12a \Ckappa \left(\|(\Heps^{\lambda})^{-1}\|_{\B(\H_{\eps})} + \|(H_0^{\lambda})^{-1}\|_{\B(\H_0)}\right)$ and where we can in addition estimate the norms $\|(\Heps^{\lambda})^{-1}\|_{\B(\H_{\eps})}$, $\|(H_0^{\lambda})^{-1}\|_{\B(\H_0)}$ by~\eqref{omezinvheps} and~\eqref{omezinvh0}.

It remains to prove~\eqref{estq2}, \ie\ to find $\tilde{\sigma}(\eps)$
and to prove that this quantity tends to zero
provided the assumptions of Theorem~\ref{Thm.main.intro} are satisfied.
This part of the proof is very technical and lengthy, on the other hand,
the reasons for our rather complicated assumptions will be explained.

\subsection{Proof of relation~\texorpdfstring{\eqref{estq2}}{estq2}}%
\label{proofrel}
Let $\phi,\psi \in W_0^{1,2}(\Omega_0)$.
We have to establish suitable estimates on the difference
of the following sesquilinear forms
\begin{align}
\label{qepssesq}Q_{\eps}^{\lambda}(\phi,\psi)=&\intom\frac{1}{hh_{\eps}}\der\bar{\phi}\der\psi \,ds\,dt+\\
\nonumber&+\frac{1}{\eps^2}\intom \frac{h}{\heps}\nabla'\bar{\phi}\cdot\nabla'\psi\,ds\,dt - \frac{E_1}{\eps^2}\intom \frac{h}{\heps}\bar{\phi}\psi\,ds\,dt
\\
\nonumber&+\,\frac{1}{2}\intom\frac{1}{\heps^2}(k_1\keps_1+k_2\keps_2)\bar{\phi}\psi\,ds\,dt - \frac{3}{4}\intom\frac{h}{\heps^3}\left((\keps_1)^2+(\keps_2)^2\right)\bar{\phi}\psi\,ds\,dt
\\
\nonumber&- \intom\frac{\der\heps}{2h\heps^2}\,(\bar{\phi}\der\psi+\der\bar{\phi}\psi)\,ds\,dt
\\
\nonumber&+\intom\frac{\left(\der\heps\right)^2}{4h\heps^3} \bar{\phi}\psi\,ds\,dt -\lambda\intom\frac{h}{\heps}\bar{\phi}\psi\,ds\,dt
\end{align}
and
\begin{align}
\label{q0sesq}Q_0^{\lambda}(\phi,\psi)=&\intom\partial_s\bar{\phi}\partial_s\psi\,ds\,dt + C_\omega\intom \dot{\theta}^2 \bar{\phi}\psi \,ds\,dt
\\
\nonumber&+\frac{1}{\eps^2}\intom\nabla'\bar{\phi}\cdot\nabla'\psi\,ds\,dt - \frac{E_1}{\eps^2}\intom \bar{\phi}\psi\,ds\,dt
\\
\nonumber&- \frac{1}{4}\intom(k_1^2+k_2^2)\bar{\phi}\psi\,ds\,dt  - \lambda \intom \bar{\phi}\psi\,ds\,dt
\,.
\end{align}
On the right hand side of the estimates
the term $\sqrt{Q_0^{\lambda}[\phi]\Qeps^{\lambda}[\psi]}$ should stand. However, due to Lemma~\ref{lemmafiner} and similar statement on $Q_0^\lambda[\phi]$ we get the inequalities
\begin{align}
Q_{\eps}^{\lambda}[\psi] &\geq \frac{1}{4}\|\der\psi\|^2 + \frac{1}{2}\left(-\lambda - 9 \Ckappa^2\right) \|\psi\|^2,\\
Q_0^\lambda[\phi] &\geq \|\partial_s\phi\|^2 + (-\lambda-\frac{\Ckappa^2}{4})\|\phi\|^2.
\end{align}
Recall that we assume $\lambda<-9\Ckappa^2$, hence in front of $\|\psi\|^2$, $\|\phi\|^2$ there stand positive numbers and we can come from estimates by $\sqrt{Q_0^{\lambda}[\phi]}$, $\sqrt{\Qeps^{\lambda}[\psi]}$ to estimates by norms $\|\psi\|^2$, $\|\phi\|^2$ or norms like $\|\der\psi\|$, $\|\partial_s\psi\|$.

Some of the estimates can be performed easily using just
the Schwarz inequality and straightforward estimates:
\begin{align}
\nonumber\intom\frac{\left(\der\heps\right)^2}{4h\heps^3}\bar{\phi}\psi\,ds\,dt &\leq C_1 \eps^2 \left(\|\dot{k}^{\eps}_1\|_{\infty} + \|\dot{k}^{\eps}_2\|_{\infty}\right)^2\sqrt{Q_0^{\lambda}[\phi]\Qeps^{\lambda}[\psi]},\\
\label{est1}\left|-\intom\frac{\der\heps}{2h\heps^2}\,(\bar{\phi}\der\psi+\der\bar{\phi}\psi)\,ds\,dt\right| &\leq C_2 \eps \left(\|\dot{k}^{\eps}_1\|_{\infty} + \|\dot{k}^{\eps}_2\|_{\infty}\right)\sqrt{Q_0^{\lambda}[\phi]\Qeps^{\lambda}[\psi]},\\
\nonumber\left|\lambda\intom\frac{h}{\heps}\bar{\phi}\psi\,ds\,dt - \lambda \intom\bar{\phi}\psi\,ds\,dt\right|&\leq C_3\eps \sqrt{Q_0^{\lambda}[\phi]\Qeps^{\lambda}[\psi]}
\end{align}
where the constants read
$C_1 := \frac{4a^2}{\Ckappa^2}$,
$C_2 := \frac{10 a}{\Ckappa}$
and $C_3 := \frac{12 a \lambda}{\Ckappa}$.
The first two inequalities are stated in terms of the quantities
$\|\dot{k}^{\eps}_\mu\|_{\infty}$,
which can be replaced by~$\|\dot{k}_\mu\|_{\infty}$
for sufficiently regular curves~$\Gamma$,
so that the results of previous papers are recovered.
On the other hand, for non-differentiable curvatures
the Steklov approximations (recall~\eqref{deraproximace})
yields
$$
\eps \left(\|\dot{k}^{\eps}_1\|_{\infty} + \|\dot{k}^{\eps}_2\|_{\infty}\right) \leq 4 \frac{\eps}{\delta(\eps)} \Ckappa
\,,
$$
where the right hand side tends to zero due to
the second assumption in~\eqref{delta.limits}.
Summing up, all the terms estimated in~\eqref{est1} tend to zero.

To estimate the rest of terms on the left hand side of~\eqref{estq2},
the Hilbert space decomposition~\eqref{hildecomp} has to be used.
The following computations
will also show why the assumption~\eqref{sigmakknule} is needed.

\subsubsection{Hilbert space decomposition
and estimates by \texorpdfstring{$\sigma_k(\delta(\eps))$}{sigmak(delta(eps))}}
Let us estimate the difference of terms on the third lines of~\eqref{qepssesq} and~\eqref{q0sesq}. It is easy to find
\begin{align}
\nonumber |q(\phi,\psi)|&:=\bigg|\frac{1}{2}\intom\frac{1}{\heps^2}(k_1\keps_1+k_2\keps_2)\bar{\phi}\psi\,ds\,dt - \frac{3}{4}\intom\frac{h}{\heps^3}\left((\keps_1)^2+(\keps_2)^2\right)\bar{\phi}\psi\,ds\,dt
\\
\nonumber& \qquad + \frac{1}{4}\intom (k_1^2+k_2^2)\bar{\phi}\psi\,ds\,dt\bigg|
\\
\label{odhq}&\leq 3\Ckappa\intom\left(|k_1-\keps_1|+|k_2-\keps_2|\right)|\phi|\,|\psi|\,ds\,dt + \eps a \Ckappa^3\|\phi\|\|\psi\|.
\end{align}
Then the first term on the last line
can be estimated by the Schwarz inequality to get a product of the term
$\left(\intom\left(|k_1-\keps_1|+|k_2-\keps_2|\right)|\phi|^2\,ds\,dt\right)^{1/2}$ and the analogous one with $\psi$ instead of $\phi$.
%$$3\Ckappa\intom\left(|k_1-\keps_1|+|k_2-\keps_2|\right)|\phi|\,|\psi|\,ds\,dt \leq 3\Ckappa \sqrt{\intom\left(|k_1-\keps_1|^2+|k_2-\keps_2|^2\right)|\phi|^2\,ds\,dt}\sqrt{\intom\left(|k_1-\keps_1|^2+|k_2-\keps_2|^2\right)|\psi|^2\,ds\,dt}.$$
However, to proceed further, we have to use the Hilbert space decomposition.
If we rewrite the function $\phi$ (and similarly $\psi$ in the other term)
as in formula~\eqref{defdecomp}, we get
\begin{multline*}
\intom\left(|k_1-\keps_1|+|k_2-\keps_2|\right)|\phi|^2\,ds\,dt
\\
= \intom\left(|k_1-\keps_1|+|k_2-\keps_2|\right)
\left(|\phi_1|^2\mathcal{J}_1^2 + 2\Re\,
\left(\bar{\phi}_1\mathcal{J}_1\phi^{\bot}\right)
+ |\phi^{\bot}|^2\right)\,ds\,dt
\end{multline*}
and we will be able to prove that the last integral tends to zero.
The term containing $|\phi_1|^2$ is
(after another estimate by the Schwarz inequality and recalling the normalization of $\mathcal{J}_1$) analogous to~\eqref{intf}
and tends to zero according to Proposition~\ref{theorconv}
and assumptions of Theorem~\ref{Thm.main.intro}.
The mixed terms vanish due to orthogonality of $\phi^{\bot}$
and $\mathcal{J}_1$ and thanks to the fact that neither
$k_\mu$ nor $k_\mu^{\eps}$ depend on the variable~$t$.
The remaining part with $|\phi^{\bot}|^2$
tends to zero according to the following ideas.

Using straightforward estimates it is possible to find
$$\Qeps^{\lambda}[\psi] \geq \frac{1}{\eps^2}\frac{1-4\eps a \Ckappa}{1+4\eps a \Ckappa}\intom |\nabla'\psi|^2\,ds\,dt - \frac{E_1}{\eps^2}\frac{1+4\eps a \Ckappa}{1-4\eps a \Ckappa}\intom|\psi|^2\,ds\,dt.$$
If we apply this inequality on $\psi^{\bot}$
and if we realize that
$\eps^{-2}\intom|\nabla'\psi^{\bot}|^2\,ds\,dt
\geq \eps^{-2} E_2 \|\psi^{\bot}\|^2$,
where $E_2$ is the second eigenvalue of the transverse Laplacian
$-\Delta_D^{\omega}$, we get
$$
\Qeps^{\lambda}[\psi^{\bot}] \geq \frac{1}{\eps^2}\left(E_2\frac{1-4\eps a \Ckappa}{1+4\eps a \Ckappa}(1-\beta) - E_1 \frac{1+4\eps a \Ckappa}{1-4\eps a \Ckappa}\right)\|\psi^{\bot}\|^2 + \frac{1}{\eps^2}\beta \frac{1-4\eps a \Ckappa}{1+4\eps a \Ckappa} \|\nabla'\psi^{\bot}\|^2.
$$
Here $\beta$ is a real parameter and from the relation $E_2>E_1$ it follows that $\beta$ can be chosen in such way that for small enough $\eps$ both coefficients in front of $\|\psi^{\bot}\|^2$ and $\|\nabla'\psi^{\bot}\|^2$ are positive. We define a constant $C_{\bot}$ such that the minimum of these two coefficient is equal to $\left(\frac{C_{\bot}}{2}\right)^{-2}$ which yields
\begin{equation}\label{bot1}
\|\psi^{\bot}\| \leq \eps \frac{C_{\bot}}{2} \sqrt{\Qeps^\lambda[\psi^{\bot}]}\leq \eps C_{\bot}\sqrt{\Qeps^\lambda[\psi]}, \qquad \|\nabla'\psi^{\bot}\|\leq \eps \frac{C_{\bot}}{2} \sqrt{\Qeps^\lambda[\psi^{\bot}]} \leq C_{\bot} \sqrt{\Qeps^\lambda[\psi]}.
\end{equation}
Using similar ideas, we would get also
\begin{equation}\label{bot2}
\|\phi^{\bot}\| \leq \eps C_{\bot} \sqrt{Q_0^\lambda[\phi]}, \qquad \|\nabla'\phi^{\bot}\|\leq \eps C_{\bot} \sqrt{Q_0^\lambda[\phi]}
\end{equation}
(for simplicity we put here the same constant $C_{\bot}$ even though the estimate could be somewhat finer).

Now we can finish the estimate on $|q(\phi,\psi)|$ as
\begin{equation}\label{estq}
|q(\phi,\psi)|
\leq \big(C_4 \sigma_k(\delta(\eps))
+ C_5 \eps\big)\sqrt{Q_0^{\lambda}[\phi]\Qeps^{\lambda}[\psi]}
\end{equation}
where $C_4$ and $C_5$ are constants depending on $\Ckappa$, $C_{\bot}$ and $a$.

The Hilbert space decomposition will be needed also
in the case of last two estimates which are technically most difficult.
The difference of terms on the second lines of~\eqref{qepssesq} and~\eqref{q0sesq} reads
$$m(\phi,\psi) := \frac{1}{\eps^2}\intom \left(\frac{h}{\heps}-1\right)\nabla'\bar{\phi}\cdot\nabla'\psi\,ds\,dt - \frac{E_1}{\eps^2}\intom \left(\frac{h}{\heps}-1\right)\bar{\phi}\psi\,ds\,dt.$$
Applying formula~\eqref{defdecomp} on $\phi$ and $\psi$,
we can divide $m(\phi,\psi)$ into four terms.
The term $m(\phi_1\mathcal{J}_1,\psi_1\mathcal{J}_1)$
is integrated by parts with respect to the transverse variable~$t$ twice;
as usual, we derive the following formula
for $\phi,\psi\in C_0^{\infty}(\Omega_0)$
and we can extend it to $W_0^{1,2}(\Omega_0)$ by density:
$$
  m(\phi_1\mathcal{J}_1,\psi_1\mathcal{J}_1)
  = \frac{1}{\eps^2}\intom \left(\frac{h}{\heps}-1\right)\bar{\phi_1}
  \psi_1\mathcal{J}_1\left[-\Delta'\mathcal{J}_1 - E_1\mathcal{J}_1\right]\,ds\,dt
 - \frac{1}{\eps^2}\intom\Delta\left(\frac{h}{\heps}-1\right)\bar{\phi_1}\psi_1\mathcal{J}_1^2\,ds\,dt
  \,.
$$
Here the first term vanishes, hence we can estimate
\begin{align*}
|m(\phi_1\mathcal{J}_1,\psi_1\mathcal{J}_1)| &\leq \frac{1}{\eps^2}\intom\left|\Delta\left(\frac{h}{\heps}-1\right)\right||\phi_1||\psi_1|\mathcal{J}_1^2\,ds\,dt
\\
&\leq 6\Ckappa\int_I\left(|k_1-\keps_1|+|k_2-\keps_2|\right)|\phi_1| |\psi_1| ds + 48\eps \Ckappa^3\int_I |\phi_1||\psi_1|ds
\\
&\leq 12\sqrt{3}\Ckappa \sigma_k(\delta(\eps))\|\phi_1\|_{W^{1,2}(I)}\|\psi_1\|_I + 48\eps \Ckappa^3 \|\phi_1\|_I\|\psi_1\|_I.
\end{align*}
Similarly the terms $m(\phi_1\mathcal{J}_1,\psi^{\bot})$ and $m(\phi^{\bot},\psi_1\mathcal{J}_1)$ are integrated by parts once to get
\begin{align}
|m(\phi_1\mathcal{J}_1,\psi^{\bot})|&\leq 8 \sqrt{3E_1}\left(\sigma_k(\delta(\eps))\|\phi_1\|_{W^{1,2}(I)} + \eps a \Ckappa \|\phi_1\|_I\right)\frac{\|\psi^{\bot}\|}{\eps},\\
|m(\psi_1\mathcal{J}_1,\phi^{\bot})| &\leq 8 \sqrt{3E_1} \left(\sigma_k(\delta(\eps))\|\psi_1\|_{W^{1,2}(I)} + \eps a \Ckappa \|\psi_1\|_I\right)\frac{\|\phi^{\bot}\|}{\eps}.
\end{align}
Finally, the estimate on $m(\phi^{\bot},\psi^{\bot})$ is straightforward,
$$|m(\phi^{\bot},\psi^{\bot})| \leq 16\eps a\Ckappa\left(\frac{\|\nabla'\phi^{\bot}\|}{\eps}\frac{\|\nabla'\psi^{\bot}\|}{\eps}+E_1\frac{\|\phi^{\bot}\|}{\eps}\frac{\|\psi^{\bot}\|}{\eps}\right)$$
where due to \eqref{bot1} and \eqref{bot2} the term in bracket is bounded.
Summing up, we get due to relations~\eqref{bot1} and~\eqref{bot2}
\begin{equation}\label{estm}
|m(\phi,\psi)|\leq \big(C_6 \sigma_k(\delta(\eps))
+ C_7 \eps\big)\sqrt{Q_0^{\lambda}[\phi]\Qeps^{\lambda}[\psi]}
\,,
\end{equation}
where the constants $C_6$, $C_7$ again depend on $\Ckappa$, $C_{\bot}$, $a$ and in addition on $E_1$.

Using the Hilbert space decomposition
we have shown that in consequence of assumption~\eqref{sigmakknule}
(which is automatically satisfied for $I$ bounded),
the terms $|q(\phi,\psi)|$ and $|m(\phi,\psi)|$
tend to zero in the limit $\eps\to 0$.
In the next section we show why also the assumption~\eqref{sigmathetaknule}
is needed.

\subsubsection{Estimates by
\texorpdfstring{$\sigma_{\dot{\theta}}(\tilde{\delta}(\eps))$}%
{sigma{\dot{theta}}(\tilde{delta}(eps))}}
Finally,
the terms on the first lines of~\eqref{qepssesq} and~\eqref{q0sesq} are estimated,
\ie\ we examine the sesquilinear form
\begin{equation}\label{defl}
l(\phi,\psi):= \intom\frac{1}{hh_{\eps}}\der\bar{\phi}\der\psi \,ds\,dt - \intom\partial_s\bar{\phi}\partial_s\psi\,ds\,dt - C_\omega\intom \dot{\theta}^2 \bar{\phi}\psi \,ds\,dt
\end{equation}
(recall $C_\omega = \|\ader\mathcal{J}_1\|^2_{L^2(\omega)}$).
We again have to decompose the functions $\psi$ and $\phi$ using~\eqref{defdecomp} which leads to long expressions where one part of the terms subtracts and other part of terms vanish when $\eps\to 0$ due to relations~\eqref{bot1} and~\eqref{bot2}.
However, also the following, problematic term occurs
\begin{equation}\label{problem}
\intom \dot{\theta}\ader\mathcal{J}_1
\left(\psi_1\partial_s\bar{\phi}^{\bot}
+\bar{\phi}_1\partial_s\psi^{\bot}\right) \,ds\,dt.
\end{equation}

This term is expected to vanish in the limit as $\eps \to 0$
due to appearance of $\phi^{\bot}$ and $\psi^{\bot}$.
However, we do not have any suitable estimate on the longitudinal derivatives
of the functions $\phi^{\bot}$ and $\psi^{\bot}$,
\ie~no vanishing control over
$\|\partial_s\bar{\phi}^{\bot}\|$ and $\|\partial_s\bar{\psi}^{\bot}\|$.
For this reason, we would like to perform an integration
by parts with respect to the variable~$s$,
which requires $\dot{\theta}$ to be differentiable.
To avoid this extra assumption,
we mollify also $\dot{\theta}$ using the Steklov approximation
$$
  \dot{\theta}^{\eps}(s):=\frac{1}{\tilde{\delta}(\eps)}\int_{s-\frac{\tilde{\delta}(\eps)}{2}}^{s+\frac{\tilde{\delta}(\eps)}{2}}\dot{\theta}(\xi)\,d\xi
  \,.
$$
Here again $\eps \mapsto \tilde{\delta}(\eps)$
is a continuous function which vanishes when $\eps\to 0$.
Then we rewrite the first term in~\eqref{problem}
in the following way and the integration by parts is justified:
\begin{align}
\nonumber&\intom \dot{\theta}\ader\mathcal{J}_1\psi_1\partial_s\bar{\phi}\,ds\,dt = \intom (\dot{\theta}-\dot{\theta}^{\eps})\ader\mathcal{J}_1\psi_1\partial_s\bar{\phi}^{\bot}\,ds\,dt + \intom \dot{\theta}^{\eps}\ader\mathcal{J}_1\psi_1\partial_s\bar{\phi}^{\bot}\,ds\,dt
\\
\label{radek1}&=\intom (\dot{\theta}-\dot{\theta}^{\eps})
\ader\mathcal{J}_1\psi_1\partial_s\bar{\phi}^{\bot}\,ds\,dt
+ \intom \big(\dot{\theta}^{\eps}\big)^{\!\!\mbox{\Large .}} \,
\ader\mathcal{J}_1\psi_1\bar{\phi}^{\bot}\,ds\,dt
+ \intom \dot{\theta}^{\eps}\ader\mathcal{J}_1(\partial_s\psi_1)
\bar{\phi}^{\bot}\,ds\,dt
\,,
\end{align}
and similarly for the second term in~\eqref{problem}.
In~\eqref{radek1} the first term will be estimated by $\sigma_{\dot{\theta}}(\tilde{\delta}(\eps))$ and tends to zero due to Proposition~\ref{theorconv}, the other terms tend to zero due to~\eqref{bot1},~\eqref{bot2} and the boundedness of $\dot{\theta}$. After tedious computations we get the final formula
\begin{equation}\label{estl}
|l(\phi,\psi)|\leq \left(C_8 \sigma_{\dot{\theta}}(\tilde{\delta}(\eps)) + C_9 \eps\right)\sqrt{Q_0^{\lambda}[\phi]\Qeps^{\lambda}[\psi]}
\end{equation}
where the constants $C_8$ and $C_9$ depend on the quantities $\Ckappa$, $C_{\bot}$, $a$, $E_1$ and also $C_\omega$, $\|\dot{\theta}\|_{\infty}$.

\subsection{Summary}
Putting all the estimates~\eqref{est1},~\eqref{estq},
\eqref{estm} and \eqref{estl} together,
we get
\begin{equation}\label{defsigmaeps}
\tilde{\sigma}(\eps):= \left(C_1 + C_2\right) \eps \left(\|\dot{k}^{\eps}_1\|_{\infty} + \|\dot{k}^{\eps}_2\|_{\infty}\right) + \left(C_3 + C_5 + C_7 + C_9\right)\eps + \left(C_4 + C_6\right)\sigma_k(\delta(\eps)) + C_8 \sigma_{\dot{\theta}}(\tilde{\delta}(\eps))
\end{equation}
which tends to zero in the limit $\eps\to 0$
due to assumptions of Theorem~\ref{Thm.main.intro}
and the proof of this theorem is in fact completed.

In more details, if we implement~\eqref{defsigmaeps} into~\eqref{castlemma1}
and if we find appropriate constant $\tilde{C}$ as the maximum of all
constants involved, we get the statement of Lemma~\ref{lemma1}.
In combination with the result of Lemma~\ref{lemmanormresh0}
and relation~\eqref{estlambda}
we can set $C:=C_{\lambda}\left(\tilde{C} + \tilde{\tilde{C}}\right)$
to get the statement of Theorem~\ref{Thm.main.intro}.
Let us note that the constant $C$
is thus function of $\lambda$, $\Ckappa$, $C_{\bot}$, $a$, $E_1$, $C_\omega$
and $\|\dot{\theta}\|_{\infty}$.

%-------------------%
\section{Conclusion}\label{seconcl}
%-------------------%
%
The objective of this paper was to establish the effective
Hamiltonian approximation~\eqref{limit} in the norm-resolvent sense
and under minimal regularity assumptions about the waveguide.
Our main result is summarized in Theorem~\ref{Thm.main.intro}.
Let us discuss its assumptions
and links with previous results here.

\subsection{Comparison with previous results}
%First of all, let us emphasize that Theorem~\ref{Thm.main.intro}
%implies all the known results mentioned in Section~\ref{Sec.why}.
As mentioned in Section~\ref{Sec.why},
the norm resolvent convergence of Theorem~\ref{Thm.main.intro}
was proved previously under sufficiently regular assumptions
about~$\Gamma$ and~$\theta$.
%and the convergence rate was given as well (see, \eg, \cite{FK4}).
Let us show that our conclusions correspond to these results.

In case when $k_1, k_2$ and $\dot{\theta}$
are Lipshitz continuous with Lipshitz constants
$L_{k_1}$, $L_{k_2}$ and $L_{\dot{\theta}}$, respectively,
then $\|\dot{k}^{\eps}_\mu\|_{\infty}\leq L_{k_\mu}$,
with $\mu = 1,2$, for any choice of $\delta(\eps)$.
Consequently, we can abandon the second condition in~\eqref{delta.limits}.
Indeed, as we explained in Section~\ref{Sec.strategy.intro},
this assumption was needed to ensure that the quantity
$\eps\|\dot{k}_{\mu}^\eps\|_{\infty}$ tends to zero
if $\|\dot{k}_{\mu}^\eps\|_{\infty}$ is not bounded,
on the other hand, for Lipshitz continuous $k_{\mu}$ this is not the case
(\cf~also the remarks below~\eqref{est1}).
Hence we can simply choose $\delta(\eps) =\tilde{\delta}(\eps)= \eps$, then $\sigma_k(\delta(\eps))\propto \left(L_{k_1} + L_{k_2}\right)\eps$ and $\sigma_{\dot{\theta}}(\tilde{\delta}(\eps))\propto L_{\dot{\theta}}\eps$,
and similarly as in Theorem~\ref{Thm.main.intro} we find
$$
  \left\|
  U(-\Delta_D^{\Omega_\varepsilon} -\varepsilon^{-2} E_1-i)^{-1}U^{-1}
  - (H_\mathrm{eff}-i)^{-1} \oplus 0^\bot
  \right\|_{\B(\H_0)}
  \leq C \eps \left( 1 + L_{k_1} + L_{k_2} + L_{\dot{\theta}}\right).
$$
In this way we get the $\eps$-type decay rate well known
from the other papers (see, \eg, \cite{FK4}).

Of course, in case of $k_1, k_2$ and $\dot{\theta}$ differentiable
with bounded derivative, $L_{k_1},L_{k_2}$ and $L_{\dot{\theta}}$
can be replaced by $\|\dot{k}_1\|_{\infty},\|\dot{k}_2\|_{\infty}$
and $\|\ddot{\theta}\|_{\infty}$, respectively.

On the other hand, Theorem~\ref{Thm.main.intro}
covers much wider class of curves than previous papers.
It is reasonable to expect a worse decay rate
if the functions $k_1, k_2$ and $\dot{\theta}$ are not differentiable.
Then a bound on the decay rate can be obtained by optimizing
the choice of~$\delta(\eps)$, as the following example shows.
\begin{example}
Let
\begin{align*}
k_1(s) &= \begin{cases}
1, &s\in (2n, 2n+1) \,, \\
-1, &s\in (2n+1,2n+2) \,,
\end{cases}
&\qquad n\in \Z \,,\\
k_2(s)&=0\,,&\forall s\in\R \,.
\end{align*}
The corresponding curve is lying in a plane and is formed by arcs of circle
with radius $1$ whose center is in one half-plane for $s\in (2n, 2n+1)$
and in the other half-plane for $s\in (2n+1,2n+2)$
(\cf~the left hand side of Figure~\ref{nefrenet} for a part of this infinite curve).
Then the best possible estimate reads
$$
  \|\dot{k}^{\eps}_1\|_{\infty} \leq \frac{2 \|k_1\|_{\infty}}{\delta(\eps)}
  = \frac{2}{\delta(\eps)}
  \,,
$$
hence on the right hand side of~\eqref{NR.bound}
the term proportional to $\eps\delta(\eps)^{-1}$ occurs.
At the same time, for this choice of $k_1,k_2$
it holds $\sigma_k(\delta(\eps))\propto \sqrt{\delta(\eps)}$.
For simplicity we will look for the optimal function $\delta(\eps)$
in the class of polynomials, here the most convenient choice is
$\delta(\eps) = \eps^{2/3}$ since then the term on the right hand side
of~\eqref{NR.bound} is proportional to $\eps^{1/3}$
(for suitable $\dot{\theta}$).
\end{example}

\subsection{Optimality of our assumptions}
Let us now turn to the optimality of the conditions under that our Theorem~\ref{Thm.main.intro} holds.

Condition~(i) of Assumption~\ref{assumgamma}
requires $\Gamma \in W^{2,\infty}_{\mathrm{loc}}(I;\R^3)$,
which seems to be the minimal condition to guarantee
that a (weakly) differentiable moving frame,
necessary for the definition of a simultaneously
twisted and bent tube along the curve, exists.
At the same time, the boundedness of curvature~$\kappa$ is
necessary to consider the waveguide even as an abstract Riemannian manifold
(\cf~Section~\ref{sectube}).
We therefore consider these hypotheses as the natural ones.
\footnote{
The case of `broken-line' waveguide or, more generally,
the question of shrinking tubular neghbourhoods of graphs
do not fit in the present setting.
We refer to recent works of Grieser \cite{Grieser-2008,Grieser}
for results and references in this field.
}
The same concerns the injectivity assumption~(iii) of Assumption~\ref{assumgamma}
if we want to interpret the waveguide as a genuine physical device
embedded~$\R^3$. But our results hold without
this last assumption, as pointed out
in Remarks~\ref{Rem.self} and~\ref{Rem.self.bis}.

On the other hand, the global boundedness of~$\dot\theta$
from Assumption~\ref{assumgamma}.(ii)
is not necessary for the definition of a non-self-intersecting waveguide.
In Figure~\ref{Fig.theta} we present an example of
an infinite twisted waveguide with elliptical cross-section
such that $\dot{\theta}(s)$ tends to infinity as $|s|\to \infty$.
It can be introduced and handled by the methods of Section~\ref{Sec.pre}
without problems. However, the form domain of the transformed
Laplacian~$\tilde{H}_\eps$ will not coincide
with the Sobolev space $W_0^{1,2}(\Omega_0)$
(\ie~\eqref{form.domain} will not hold)
and an extensive modification of the present strategy
to get the operator limit~\eqref{limit} would be required.

\begin{figure}[h!]
\begin{center}
\includegraphics[width=0.6\textwidth]{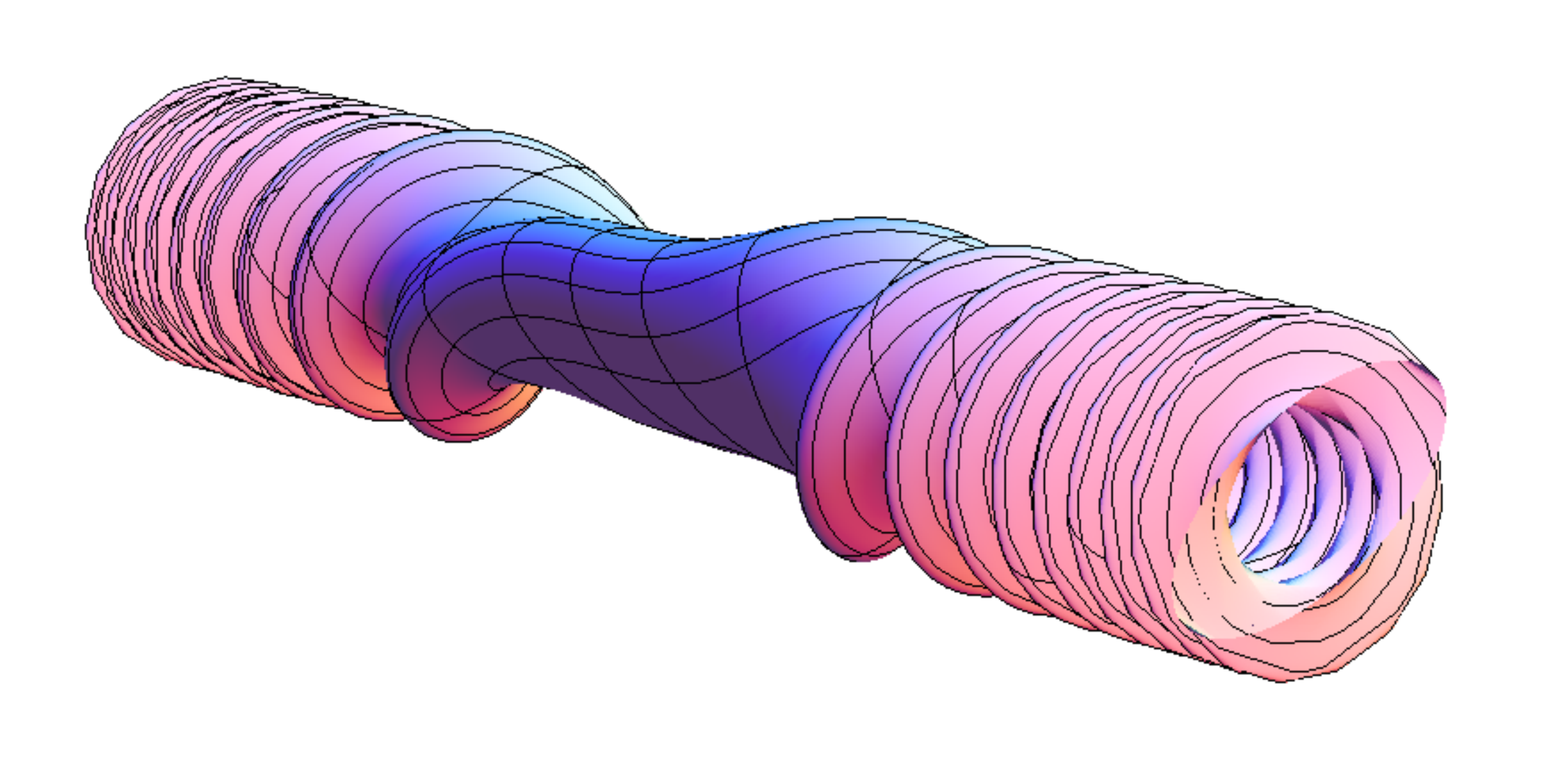}%
\caption{A waveguide of elliptical cross-section
along a straight line whose twisting diverges at infinity.
The embedded cylindrical channel is responsible for the existence
of essential spectrum.}\label{Fig.theta}
\end{center}
\end{figure}
It is also important to emphasize that the unboundedness of~$\dot{\theta}$
may lead to pathological spectral properties.
Indeed, the condition $\dot{\theta}(s)\to\infty$ as $|s|\to \infty$
implies that~$H_\mathrm{eff}$ has purely discrete spectrum,
while $\sigma_\mathrm{ess}(-\Delta_D^{\Omega_\eps})\not=\varnothing$
in general.
Actually, the latter happens whenever the cross-section~$\omega$
contains the origin of~$\R^2$ (as in Figure~\ref{Fig.theta}),
so that there is an infinite cylindrical channel in~$\Omega_\eps$
leading to scattering waves.
It does not contradict the validity of the effective
Hamiltonian approximation in principle, since the threshold
of the essential spectrum of $-\Delta_D^{\Omega_\eps}-\eps^{-2}E_1$
tends to infinity in the limit as $\eps \to 0$,
but the usefulness of the approximation becomes doubtful.

We admit that Assumption~\ref{ass2} seems unnatural
and it is true that it comes from our technical procedure
of mollifying the curvature functions~$k_1,k_2$ and~$\dot{\theta}$.
Although it covers a wide class of waveguides
and, in particular, all the previously known results,
there are still reasonable situations for which
Assumption~\ref{ass2} does not hold,
as the following counterexample shows.
\begin{example}\label{excounter}
Let us define the curve $\Gamma^{\mathrm{osc}}:\R\to\R^3$
by giving its curvatures:
\begin{align*}
k_1^{\mathrm{osc}}(s) &:= \begin{cases} \;1,
 &s\in\Big((n-1+\frac{2k}{2n})\pi,(n-1+\frac{2k+1}{2n})\pi\Big),\\
                     -1,
 &s\in\Big((n-1+\frac{2k+1}{2n})\pi,(n-1+\frac{2k+2}{2n})\pi\Big),
                     \end{cases} \quad n\in \N,\;k = 0,1...,n-1,\\
k_2^{\mathrm{osc}}(s) &:= 0.
\end{align*}

\begin{figure}[h]
\begin{center}
\includegraphics[width=15 cm]{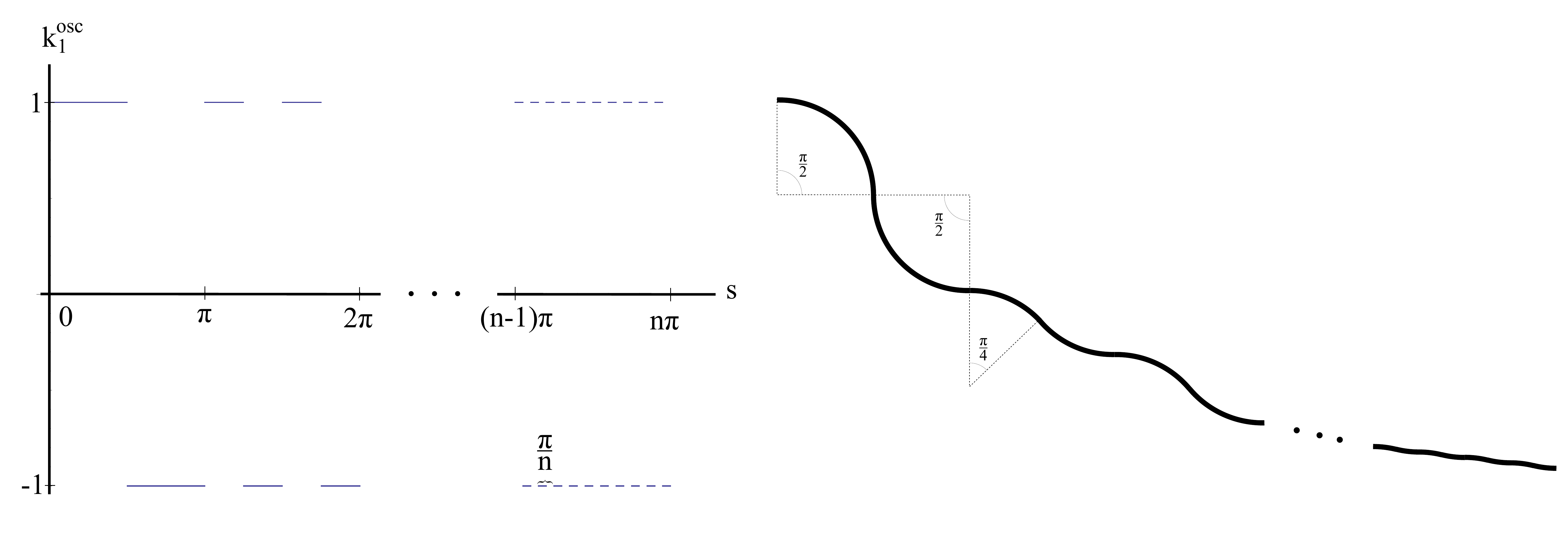}
\caption{The counterexample curve.}\label{counter}
\end{center}
\end{figure}

The graph of the function $k_1^{\mathrm{osc}}$
is given in Figure~\ref{counter}
as well as $\Gamma^{\mathrm{osc}}$ itself.
This curve lies in a plane and consists of arcs of circle
of radius $1$ which are shorter and shorter as $s$ grows.
For $s\to \infty$, this curve looks like a straight line,
however, the curvature is still nonzero.
It is possible to show that for all $\eps>0$
there exists $n_0\in \R$ such that
$$
  \sup_{|\eta|\leq \frac{\delta(\eps)}{2}}
  \int_{(n-1)\pi}^{n\pi}
  \left|k_1^{\mathrm{osc}}(s)-k_1^{\mathrm{osc}}(s+\eta)\right|^2 ds
  = 4\pi
  \,,\qquad \forall n\in\N,\;n\geq n_0 ,
$$
and this holds true for any choice of function $\delta$.
Hence the curvature $k_1^{\mathrm{osc}}$ does not satisfy
Assumption~\ref{ass2} and our Theorem~\ref{Thm.main.intro}
does not apply.

To get at least some information about the dynamics
in such a pathological quantum waveguide,
we examine the spectrum of the Dirichlet Laplacian
$-\Delta_D^{\Omega_\eps^{\mathrm{osc}}}$
in a untwisted tube with rectangular cross-section
constructed along the curve~$\Gamma^{\mathrm{osc}}$.
It is possible to show that for any positive~$\eps$
(such that the tube does not overlap itself)
$$
  \inf{
  \sigma_{\mathrm{ess}}\left(
  -\Delta_D^{\Omega_\eps^{\mathrm{osc}}}
  -\frac{E_1}{\eps^2}
  \right)
  }
  \geq 0
  \,.
$$
On the other hand, assuming that the effective dynamics
is governed by~\eqref{defheff}, we have
$$
  \sigma(\Heff) = \sigma_\mathrm{ess}(\Heff)
  = \sigma\Big(-\Delta_D^\R-\frac{1}{4}\Big)
  = \Big[-\frac{1}{4},\infty\Big)
  .
$$
That it, $[-\frac{1}{4},0]$ belongs to
the essential spectrum of~$\Heff$,
while the threshold of the essential spectrum
of the three-dimensional renormalized Hamiltonian is non-negative
for every positive~$\eps$.

Again, this pathological spectral behavior does not
necessarily imply that the norm-resolvent convergence of
$-\Delta_D^{\Omega_\eps^{\mathrm{osc}}} - \eps^{-2}E_1$
to~$\Heff$ does not hold.
As a matter of fact, it is possible to show that
$-\Delta_D^{\Omega_\eps^{\mathrm{osc}}} - \eps^{-2}E_1$
possesses an infinite number of negative eigenvalues,
hence it may happen that these eigenvalues cover
the whole interval $[-\frac{1}{4}, 0]$
in the limit as $\eps\to 0$.
\end{example}

In any case, we would like to emphasize that our Theorem~\ref{Thm.main.intro}
represents the first norm-resolvent convergence result
for unbounded waveguides in the full setting of bending and twisting .
The question of optimality of Assumption~\ref{ass2}
in the unbounded case remains open.

\subsection{Two-dimensional waveguides}
The methods of the present paper also enable one to improve
the known results \cite{DE,FK4} about the effective Hamiltonian
approximation in strip-like neighbourhoods of plane curves.
The norm-resolvent convergence in the two-dimensional case
does not follow directly from our three-dimensional Theorem~\ref{Thm.main.intro},
but it can be established exactly in the same way.
The proof is in fact much simpler because there is
just one curvature function, the Frenet frame always exists
(it coincides with a relatively parallel frame)
and there is no twisting
(codimension of the reference curve is one).
Here we therefore present just the ultimate result without proof.
The interested reader who is not willing to adapt the present proof
to the two-dimensional case himself/herself
is referred to~\cite{Helenka-diplomka}.

Let $\Gamma:I\to\R^2$ be a unit-speed $C^1$-smooth curve,
where~$I$ is an arbitrary open interval
(finite, semi-infinite, infinite).
The vector fields $T:=\dot{\Gamma}$ and $N:=(-\dot\Gamma^2,\Gamma^1)$
form a positively oriented Frenet frame of~$\Gamma$.
We introduce the curvature function~$\kappa$ by
the Serret-Frenet formula $\dot{T}=-\kappa N$.
Note that, contrary to the three-dimensional case,
$\kappa$~is allowed to change sign
(and the value of the sign depends on the parametrization).

\begin{figure}[h!]
\begin{center}
\includegraphics[width=0.6\textwidth]{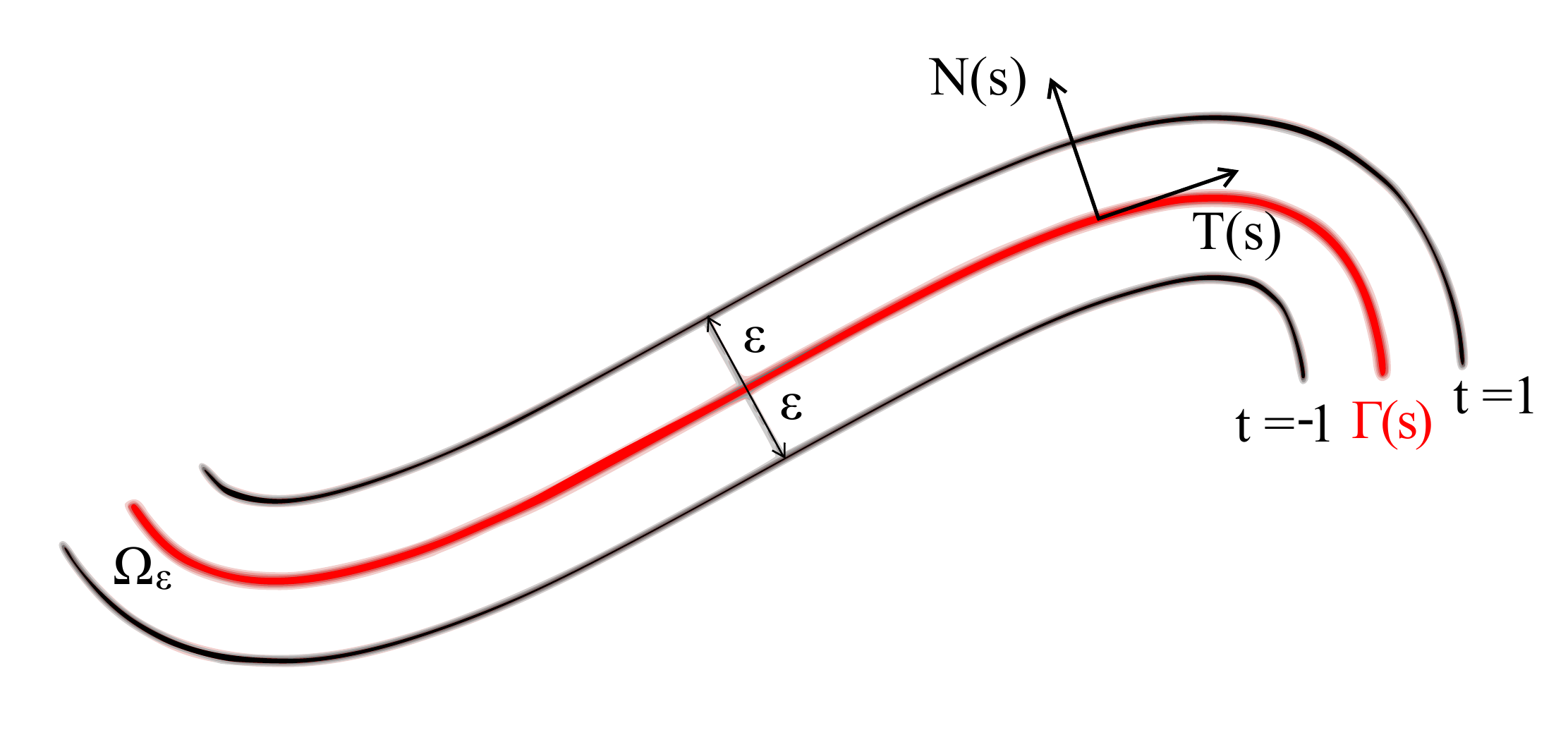}
\caption{A two-dimensional quantum waveguide.}\label{fig2D}
\end{center}
\end{figure}

In analogy with~\eqref{tube.image} and~\eqref{zavedeniL},
the two-dimensional waveguide~$\Omega_\eps$ is introduced
as the image~\eqref{tube.image} where the mapping~$\mathcal{L}$
is given now by
$$
  \mathcal{L}(s,t) := \Gamma(s) + \eps \, t \, N(s)
  \,,
$$
with $(s,t) \in I \times (-1,1)$ (see Figure~\ref{fig2D}).
The unitary transforms~\eqref{U1} and~\eqref{U2}
should be replaced by
\begin{align*}
  U_1 &:
  L^2(\Omega_\varepsilon)
  \to L^2\big(
  \Omega_0,\varepsilon\;\!h(s,t)\;\!ds\;\!dt
  \big):
  \{\psi\mapsto\psi\circ\mathcal{L}^{-1}\}
  \,,
  \\
  U_2 &:
  L^2\big(
  \Omega_0,\varepsilon^2\;\!h(s,t)\;\!ds\;\!dt
  \big)
  \to L^2(\Omega_0) :
  \{\psi \mapsto \sqrt{\eps h} \;\! \psi\}
  \,,
\end{align*}
where $h(s,t) := 1-\eps\kappa(s) t$,
and we again define $U:=U_2 U_1$.
The latter enables one to approximate in the limit as $\eps \to 0$
the Dirichlet Laplacian in~$\Omega_\eps$
by the well known one-dimensional effective Hamiltonian
$$
  \Heff :=  - \Delta_D^{I}  - \frac{\kappa^2}{4}
  \,.
$$
The main idea of the proof again consists in replacing~$U_2$
by~\eqref{U2.tilde} using the mollification~$\kappa^{\eps}$
defined analogously to~\eqref{aproximace}.

The two-dimensional version of Theorem~\ref{Thm.main.intro}
reads as follows.
\begin{theorem}\label{Thm.2D}
Let the following assumptions hold true:
\begin{enumerate}
\item[\emph{(i)}]
$\Gamma\in W^{2,\infty}_{\mathrm{loc}}(I;\R^2)$
\ and \
$\kappa\in L^\infty(I)$.
\item[\emph{(ii)}]
$\Omega_{\eps}$ does not overlap itself
(\ie~$\mathcal{L}$ is injective)
for small enough~$\eps$.
\item[\emph{(iii)}]
$
\displaystyle
\lim_{\eps\to 0} \sigma_\kappa(\delta(\eps)) = 0
$
for some positive continuous function $\delta$ satisfying~\eqref{deltaknule},
where~$\sigma_\kappa$ is defined by~\eqref{sigmaf}.
\end{enumerate}
Then there exist positive constants~$\eps_0$ and~$C$
such that for all $\eps \leq \eps_0$,
\begin{equation*}%\label{NR.bound}
  \left\|
  U(-\Delta_D^{\Omega_\varepsilon} -\varepsilon^{-2} E_1-i)^{-1}U^{-1}
  - (H_\mathrm{eff}-i)^{-1} \oplus 0^\bot
  \right\|_{\mathcal{B}(L^2(\Omega_0))}
  \\
  \leq C \, \Big(
  \eps + \eps \;\! \|\dot{\kappa}^\eps\| + \sigma_\kappa(\delta(\eps))
  \Big)
  \,,
\end{equation*}
where~$0^\bot$ denotes the zero operator on the orthogonal
complement of the span of
$
  \{\varphi\otimes\mathcal{J}_1 \,|\, \varphi \in L^2(I)\}
$ and $U=U_2U_1$.
\end{theorem}

\subsection{Different boundary conditions}
It is well known that the structure of the effective Hamiltonian
in the limit~\eqref{limit} is a consequence
of the choice of Dirichlet boundary conditions
on the `lateral boundary' $I\times\partial\omega$.
Indeed, there is no geometric potential for Neumann
boundary conditions \cite{Schatzman_1996}
and the limit can have a completely
different nature if one considers a combination
of Dirichlet and Neuman boundary conditions \cite{K5}.
On the other hand, our choice of Dirichlet boundary
conditions on $(\partial I)\times\omega$ is not essential.
In the same manner, we could impose any other type
of boundary conditions (Neumann, Robin, periodic, \etc),
which would lead to an analogue of Theorem~\ref{Thm.main.intro},
with the boundary conditions of~$H_\mathrm{eff}$ changed accordingly.
In particular, the choice of periodic boundary conditions
enable us to cover the case of tubes about closed (compact) curves.

\subsection{Non-thin waveguides}
Finally, let us mention that the tricks of the present paper
how to deal with quantum waveguides under mild regularity assumptions
do not restrict to the effective Hamiltonian approximation.
For instance, the usage of the relatively parallel frame
instead of the Frenet frame enables one to extend
some spectral results for $-\Delta_D^{\Omega_\eps}$
(with~$\eps$ not necessarily small)
to the general case of waveguides along merely twice differentiable
curves with possibly vanishing curvature.
In particular, we have in mind the classical results
about the curvature-induced bound states \cite{DE,ChDFK}
and the recent ones about Hardy-type inequalities
due to twisting \cite{EKK,K6-with-erratum,KZ1}.
\subsection*{Acknowledgement}
We are grateful to Denis Borisov for valuable discussions,
in particular for his idea how to significantly relax
our initial hypotheses about the waveguide regularity.
This work has been partially supported by
the Czech Ministry of Education, Youth, and Sports
within the project LC06002 and by the GACR grant No.~P203/11/0701.

%\newpage
%\bibliographystyle{plain}
%\bibliography{vyzkumak}
%\bibliographystyle{amsplain}
%\bibliography{bib}
\addcontentsline{toc}{section}{References}

\providecommand{\bysame}{\leavevmode\hbox to3em{\hrulefill}\thinspace}
\providecommand{\MR}{\relax\ifhmode\unskip\space\fi MR }
% \MRhref is called by the amsart/book/proc definition of \MR.
\providecommand{\MRhref}[2]{%
  \href{http://www.ams.org/mathscinet-getitem?mr=#1}{#2}
}
\providecommand{\href}[2]{#2}

\end{document}